    \documentclass[11pt]{amsart}

        \usepackage{latexsym}
        \usepackage{amssymb}
        \usepackage{amsmath}
        \usepackage{amsfonts}
        \usepackage{amsthm}
        \usepackage[colorlinks=true]{hyperref}
        \ifpdf
          \usepackage[all,knot,poly,2cell,cmtip]{xy}
        \else
          \usepackage[all,knot,poly,2cell,dvips,cmtip]{xy}
        \fi
             \UseAllTwocells
        \usepackage{xspace}

        \usepackage{eucal}
        \usepackage{stmaryrd}
        \usepackage[T1]{fontenc}

        \numberwithin{equation}{section}
        \theoremstyle{plain}
        \newtheorem{theorem}[equation]{Theorem}
        \newtheorem{corollary}[equation]{Corollary}
        \newtheorem{lemma}[equation]{Lemma}
        
        \newtheorem{proposition}[equation]{Proposition}
        \newtheorem{maintheorem}{Theorem}

        \theoremstyle{definition}
        \newtheorem{definition}[equation]{Definition}
        \newtheorem{example}[equation]{Example}
        
        \newtheorem*{example*}{Example}

        \theoremstyle{remark}
        \newtheorem{remark}[equation]{Remark}
        \newtheorem*{remark*}{Remark}
        
        \newtheorem*{claim*}{Claim}
        
        \makeatletter
        \def\namedlabel#1#2{\begingroup
        \def\@currentlabel{#2}%
        \label{#1}\endgroup
        }
        \makeatother

        \newcommand{\suchthat}{\,:\,}
        \newcommand{\itemref}[1]{\eqref{#1}}

        \newcommand{\C}{\mathbb{C}}
        \newcommand{\Z}{\mathbb{Z}}

        \newcommand{\Q}{\mathbb{Q}}

        \newcommand{\Orb}{\mathcal{O}}   
       \DeclareMathOperator{\spec}{Spec} 

           \newcommand{\Aff}{\mathbb{A}}

        \newcommand{\MOD}{\mathsf{Mod}}    

        \DeclareMathOperator{\Def}{Def}

        \DeclareMathOperator{\coker}{coker}
         \DeclareMathOperator{\Hom}{Hom}
        \DeclareMathOperator{\Ext}{Ext}
        \DeclareMathOperator{\Tor}{Tor}
        
        \newcommand{\COHO}[1]{\mathcal{H}^{{#1}}}
        \newcommand{\trunc}[1]{\tau^{{#1}}}
        \newcommand{\RDERF}{\mathsf{R}}
        \newcommand{\LDERF}{\mathsf{L}}
        
        \newcommand{\DCAT}{\mathsf{D}}
        \newcommand{\RHom}{\RDERF\!\Hom}
        \newcommand{\SHom}{\mathcal{H}om}
        \newcommand{\SRHom}{\RDERF\SHom}
        \newcommand{\QCOH}{\mathsf{QCoh}}
        \newcommand{\COH}{\mathsf{Coh}}
        \newcommand{\PERF}{\mathsf{Perf}}

        \renewcommand{\bar}[1]{\overline{{#1}}}

        \DeclareMathOperator{\Aut}{Aut}
        \DeclareMathOperator{\End}{End}

        \newcommand{\tensor}{\otimes}

        \newcommand{\opp}{\circ}
        
\numberwithin{equation}{section}

\newcommand{\qcsubscript}{\mathrm{qc}} 

\newcommand{\DQCOH}[1][]{\DCAT_{\qcsubscript{#1}}} 

\newcommand{\GL}{\mathrm{GL}} 
\newcommand{\Ga}{\mathbb{G}_a} 

\newcommand{\hocolim}[1]{\underset{#1}{\mathrm{hocolim}}\,}

\newcommand{\shfcoho}{\mathrm{H}}

\newcommand{\Vect}{\mathsf{Vect}}

\newcommand{\spref}[1]{\href{http://stacks.math.columbia.edu/tag/#1}{#1}}

\makeatletter
\newcommand{\labitem}[2]{%
\def\@itemlabel{(\textbf{#1})}
\item
\def\@currentlabel{\textbf{#1}}\label{#2}}
\makeatother

\title[Bondal--Orlov for Deligne--Mumford stacks]{A generalized Bondal--Orlov full faithfulness criterion for Deligne--Mumford stacks}
\date{May 9, 2024}
\author[J. Hall]{Jack Hall}
\email{jack.hall@unimelb.edu.au}
\address{School of Mathematics \& Statistics\\
         The University of Melbourne\\
         Parkville, VIC, 3010, Australia}
\author[K. Priver]{Kyle Priver}


\theoremstyle{plain}
\newtheorem*{theorem*}{Theorem}
\renewcommand{\mathbb}{\mathbf}
    \newcommand{\Lift}{\mathsf{Lift}}
    \usepackage{enumitem}
    \setlist[enumerate]{font=\upshape}

\begin{document}
\maketitle
\begin{abstract}
  Let $X$, $Y$ be smooth projective varieties over $\mathbf{C}$. Let $K$
be a bounded complex of coherent sheaves on $X\times Y$ and let
$\Phi_K \colon \mathsf{D}^b_{\mathsf{Coh}}(X) \to
\mathsf{D}^b_{\mathsf{Coh}}(Y)$ be the resulting Fourier--Mukai
functor. There is a well-known criterion due to Bondal--Orlov for
$\Phi_K$ to be fully faithful. This criterion was recently extended to
smooth Deligne--Mumford stacks with projective coarse moduli schemes
by Lim--Polischuk. We extend this to all smooth, proper
Deligne--Mumford stacks over arbitrary fields of characteristic
$0$. Along the way, we establish a number of foundational results for
bounded derived categories of proper and tame morphisms of noetherian
algebraic stacks (e.g., coherent duality).

\end{abstract}
\section{Introduction}
Let $X$ be a noetherian scheme, algebraic space, or algebraic
stack. Let $F$ be a triangulated functor from $\DCAT^{b}_{\COH}(X)$,
the bounded derived category of coherent sheaves on $X$, to a
triangulated category $\mathsf{T}$. There has been extensive study of
full faithfulness criteria for such $F$. The most famous is likely the
criterion due to Bondal and Orlov (see
\cite[Thm.~1.1]{bo-semiorthogonal} and \cite[Thm.~5.1]{MR1651025}) for
smooth projective schemes over an algebraically closed field of
characteristic $0$. Our main result is a generalization of this to
Deligne--Mumford stacks.
\begin{maintheorem}\label{T:bo-field}
  Let $X$ and $Y$ be smooth, proper, Deligne--Mumford stacks over a
  field $k$ of characteristic $0$.  Let
  $K \in \DCAT^b_{\COH}(X\times_kY)$ and let
  \[
    F=\Phi_K \colon \DCAT^b_{\COH}(X) \to \DCAT^b_{\COH}(Y) \colon M \mapsto \RDERF q_*(K \tensor^{\LDERF}_{\Orb_{X\times_k Y}} \LDERF p^*M),
  \]
  be the resulting Fourer--Mukai functor, where
  $p \colon X\times_k Y \to X$ and $q \colon X\times_k Y \to Y$ are
  the projections. Then $F$ is fully faithful if and only if for each
  pair of generalized closed points $(x,\xi)$, $(x,\xi')$ of $X$:
  \[
    \Hom_{\Orb_{X}}(\kappa(x,\xi),\kappa(x,\xi')) \simeq \Hom_{\Orb_{Y}}(F(\kappa(x,\xi)),F(\kappa(x,\xi')));
  \]
  and for each pair of generalized closed points $(x_1,\xi_1)$,
  $(x_2,\xi_2)$ of $X$ and integer $i$:
    \[
      \Hom_{\Orb_{Y}}(F(\kappa(x_1,\xi_1)),F(\kappa(x_2,\xi_2))[i]) =0,
    \]
    unless $x_1\simeq x_2$ and $0\leq i \leq \dim (X)$.
\end{maintheorem}

A variant of Theorem \ref{T:bo-field} was recently proved for
Deligne--Mumford stacks with projective coarse moduli space over an
algebraically closed field of characteristic $0$ \cite{MR4280492},
provided left and right adjoints to the functor $F$ exist and have
certain properties. 
In particular, Theorem \ref{T:bo-field} is even new for algebraic spaces.

While \cite{MR4280492} used non-commutative geometry to reduce to the
orbifold case and then argued using Hilbert schemes, we proceed more closely to Bridgeland's approach. This requires some care with \emph{generalized points} (see
\S\ref{S:gen-points}), which were introduced in \cite{MR4280492} over
algebraically closed fields. A generalized point of an algebraic stack
simultaneously generalizes the residue field of an algebraic space and
an irreducible representation of a group. More precisely, a
generalized point of a quasiseparated algebraic stack $X$ is a pair
$(x,\xi)$ such that $x\colon \spec l\to X$ is a morphism, where $l$ is
a field, and $\xi$ is a simple object of the category of quasicoherent
sheaves on the residual gerbe $i_x \colon \mathcal{G}_x \subseteq X$
associated to $x$. We set $\kappa(x,\xi) = (i_x)_*\xi \in \QCOH(X)$.

Along the way to proving Theorem \ref{T:bo-field} we establish a
number of foundational results such as coherent duality (Theorem
\ref{T:coherent-duality} and Corollary \ref{C:gorenstein-duality}) and
finiteness results (Theorem \ref{T:boundedness-stacks}) for algebraic
stacks with finite diagonal. We use these results---together with
ideas of \cite{MR2964634}---to establish that full faithfulness in
families is an open condition (Corollary \ref{C:ff-open}).

We also establish a number of results on generalized points
(\S\ref{S:gen-points}). We show they form a \emph{spanning class}---in
the sense of Bridgeland \cite[Defn.~2.1]{MR1651025}---for noetherian
algebraic stacks that satisfy the Thomason condition (Corollary
\ref{C:spanning}). By \cite[Thm.~A]{hall2022remarks}, this generalizes
\cite[Prop.~2.1]{MR4280492} to smooth algebraic stacks with
quasi-affine diagonal and also many others
\cite{perfect_complexes_stacks}.

In Appendix \ref{A:retracted-covers}, we establish a criterion for an
algebraic stack to be cohomologically affine (Proposition
\ref{P:retract-cover}) and use this to characterize injectives and
projectives on such stacks (Corollary \ref{C:loc-proj-inj-glob}). This
is a technical result that is used to understand generalized points
for stacks with infinite affine stabilizers.

In Appendix \ref{A:strong-gens}, we establish an analog of
\cite[Thm.~1.1/A.1]{MR1996800} for tame algebraic stacks. While this
is not used in this article, it is closely related to the saturation
and finiteness results discussed in \S\ref{S:finite} (e.g., Corollary
\ref{C:saturated}). It follows relatively easily from ideas of
\cite{Neeman_Approx,aoki2020quasiexcellence,MR3674218} on descendable
morphisms.
\subsection{Acknowledgements}
The first author would like to thank Christian Haesemeyer, Fei Peng, and Amnon Neeman
for productive suggestions, discussions, and encouragement.
\subsection{Assumptions and conventions}
We adhere to the conventions of \cite{stacks-project}; in particular,
we have no separation hypotheses on our algebraic stacks. An algebraic
stack is \emph{locally noetherian} if it admits a smooth cover by a
locally noetherian scheme. An algebraic stack is \emph{noetherian} if
it is quasicompact, quasiseparated, and locally noetherian.

An algebraic stack is said to have the \emph{resolution property} if
it is quasicompact with affine diagonal and every quasicoherent
sheaf is a quotient of a direct sum of vector bundles of finite rank;
by the Totaro--Gross Theorem \cite{MR2108211,2013arXiv1306.5418G},
this is equivalent to being of the form $[U/\GL_n]$, where $U$ is a
quasi-affine scheme.

A flat and affine group scheme of finite presentation over a base $S$
is \emph{linearly reductive} if $\pi\colon BG \to S$ is
\emph{cohomologically affine} \cite{2008arXiv0804.2242A}; that is,
$\pi_* \colon \QCOH(BG) \to \QCOH(S)$ is exact (also see Appendix
\ref{A:retracted-covers}).

A quasicompact and quasiseparated algebraic stack $X$ has
\emph{finite cohomological dimension} if there exists $d_0$ such that
$\shfcoho^d(X,F) = 0$ for all $d>d_0$ and quasicoherent sheaves
$F$. This holds for quasicompact and quasiseparated schemes,
algebraic spaces, tame stacks, and stacks with affine stabilizers in
characteristic $0$
\cite{perfect_complexes_stacks,hallj_dary_alg_groups_classifying}. If
$X$ has finite cohomological dimension, then its unbounded derived
category can have similar properties to the unbounded derived category
of a quasicompact and quasiseparated scheme. For example, the
compact objects $\DQCOH(X)^c$ coincide with the perfect objects
$\PERF(X)$ \cite[Rem.~4.6]{perfect_complexes_stacks} (there is always
the inclusion $\DQCOH(X)^c \subseteq \PERF(X)$).

A morphism of algebraic stacks $f \colon X \to S$ is
\emph{concentrated} if for every morphism $\spec A \to S$ the
algebraic stack $X\times_S \spec A$ has finite cohomological
dimension. Quasicompact and quasiseparated representable/tame
morphisms are concentrated, as are $\Q$-stacks with affine
stabilizers. Due to subtleties with the lisse-\'etale site for
algebraic stacks, we define
$\RDERF f_{\qcsubscript,*} \colon \DQCOH(X) \to \DQCOH(S)$ as a right
adjoint to a certain $\LDERF f^* \colon \DQCOH(S) \to \DQCOH(X)$. If
$f$ is concentrated, then $\RDERF f_{\qcsubscript,*}$ behaves a lot
like familiar $\RDERF f_*$ for schemes. More detail is available in 
\cite[\S2]{perfect_complexes_stacks}.

\section{Finiteness results}\label{S:finite}
Recall that if $f \colon X \to S$ is a proper (resp.~proper and tame)
morphism of noetherian algebraic stacks, then the restriction of
$\RDERF f_{\qcsubscript,*} \colon \DQCOH(X) \to \DQCOH(S)$ to
$\DCAT^+_{\COH}(X)$ (resp.~$\DCAT^b_{\COH}(X)$) factors through
$\DCAT^+_{\COH}(S)$ (resp.~$\DCAT^b_{\COH}(S)$)
\cite[Thm.~1.2]{MR2183251}.
\begin{remark}
  There are non-noetherian versions of this
  for schemes \cite{MR0382280}. If $f$ is flat, then there are also
  non-noetherian finiteness theorems for morphisms of stacks (e.g.,
  argue as in \cite[Tag \spref{0A1H}]{stacks-project}). There are also
  variants in spectral algebraic geometry (e.g.,
  \cite{lurie_sag}). There is also a variant in the non-tame
  situation. We only prove the noetherian case.
\end{remark}
\begin{proposition}\label{P:necessary}
  Let $S$ be a noetherian algebraic stack. Let $f\colon X \to S$ be a proper
  morphism of algebraic stacks with finite diagonal. Let
  $M \in \DCAT^-_{\COH}(X)$ (resp.~$\DCAT^b_{\COH}(X)$). If
  $G \in \DQCOH(X)^c$, then
  $\RDERF f_*\SRHom_{\Orb_X}(G,M) \in \DCAT^-_{\COH}(S)$
  (resp.~$\DCAT^b_{\COH}(S)$).
\end{proposition}
\begin{proof}
  Since $S$ is noetherian, this is local on $S$
  \cite[Ex.~3.9]{perfect_complexes_stacks}, so we may assume that
  $S=\spec A$ is affine. If $X$ is tame, then the result follows from
  the usual finiteness result:
  \[
    \RHom_{\Orb_X}(G,M) \simeq \RDERF \Gamma(X,\SRHom_{\Orb_X}(G,M))
    \simeq \RDERF \Gamma(X,G^\vee \otimes^{\LDERF}_{\Orb_X} M).
  \]
  Without tameness this even shows that if
  $M\in \DCAT^b_{\COH}(X)$, then $\RHom_{\Orb_X}(G,M)$ belongs to
  $\DCAT^+_{\COH}(A)$. In general,
  \cite[Lem.~4.5(3)]{perfect_complexes_stacks} implies that there
  exists $r\geq 0$ such that for all $j\in \Z$ and $M \in \DQCOH(X)$:
  \[
    \trunc{\geq j}\RHom_{\Orb_X}(G,M) \simeq \trunc{\geq
      j}\RHom_{\Orb_X}(G,\trunc{\geq j-r}M).
  \]
  Combining this with the above gives the claim.
\end{proof}
\begin{corollary}\label{C:perfect-tor-dim}
  Let $S$ be a noetherian algebraic stack. Let $f\colon X \to S$ be a
  proper morphism of algebraic stacks of finite tor-dimension with
  finite diagonal. If $G \in \DQCOH(X)^c$, then
  $\RDERF f_*G \in \PERF(S)$.
\end{corollary}
\begin{proof}
  Again, we may assume that $S=\spec A$ is affine. First note that if
  $G \in \DQCOH(X)^c$, then it is perfect
  \cite[Lem.~4.4(1)]{perfect_complexes_stacks}. In particular,
  $G^\vee \otimes_{\Orb_X}^{\LDERF} G \otimes_{\Orb_X}^{\LDERF} G^\vee
  \in \DQCOH(X)^c$ \cite[Lem.~4.4(2)]{perfect_complexes_stacks}. Since
  $G^\vee$ is a direct summand of
  $G^\vee \otimes_{\Orb_X}^{\LDERF} G \otimes_{\Orb_X}^{\LDERF}
  G^\vee$, $G^\vee \in \DQCOH(X)^c$. By Proposition \ref{P:necessary},
  it now follows that
  $\RDERF\Gamma(X,G) \simeq \RHom_{\Orb_X}(G^\vee,\Orb_X) \in
  \DCAT^b_{\COH}(A)$. Hence, it remains to show that $\RDERF\Gamma(X,G)$ has
  finite tor-dimension \cite[Tag \spref{0658}]{stacks-project}. Since
  $G$ is perfect over $X$ and $X$ has finite tor-dimension, then $G$
  has tor-amplitude in $[a,b]$ for some $-\infty< a<b < \infty$ over
  $A$. By \cite[Prop.~4.11]{perfect_complexes_stacks}, if $I$ is an
  $A$-module, then
  \[
    \RDERF\Gamma(X,G) \otimes^{\LDERF}_A I \simeq \RDERF \Gamma(X, G \tensor^{\LDERF}_A I) \in \DCAT^{[a,b+r]}(A),
  \]
  where $r$ is chosen as in
  \cite[Lem.~4.5(2)]{perfect_complexes_stacks} for
  $\mathcal{P} = G^\vee$, which gives the claim.
\end{proof}
We now establish a converse to Proposition \ref{P:necessary}, which
extends \cite[Tag \spref{0CTT}]{stacks-project} from algebraic spaces
to algebraic stacks with finite diagonal. This will be used to
establish coherent duality for proper morphisms algebraic stacks
(Theorem \ref{T:coherent-duality}). Somewhat surprisingly, the
non-noetherian result can also be established. For background material
on pseudocoherence, we refer the interested reader to
\cite[\S3]{GAGA_theorems} and the associated references. The key
examples to keep in mind are if $X$ is noetherian:
\begin{itemize}
\item $M\in \DQCOH(X)$ is
  pseudocoherent if and only if it belongs to $\DCAT^-_{\COH}(X)$
  \cite[Cor.~I.3.5]{MR0354655}; 
\item $Y \to X$ is pseudocoherent if and only if it is locally of
  finite type \cite[Tag \spref{06BX}]{stacks-project}.
\end{itemize}
\begin{theorem}\label{T:boundedness-stacks}
  Let $A$ be a ring. Let $X \to \spec A$ be a finitely presented
  morphism of algebraic stacks with finite diagonal. Let
  $M \in \DQCOH(X)$. If $\RHom_{\Orb_X}(G,M) \in \DCAT(A)$ is
  pseudocoherent (resp.~pseudocoherent and bounded) for all
  $G \in \DQCOH(X)^c$, then $M$ is pseudocoherent
  (resp.~pseudocoherent and bounded) relative to $A$. In particular,
  if $X \to \spec A$ is pseudocoherent, then $M$ is pseudocoherent
  (resp.~pseudocoherent and bounded). 
\end{theorem}
\begin{proof}
  The latter statement follows from the former and \cite[Tag
  \spref{0DHQ}]{stacks-project}.

  By noetherian approximation, there is a finitely generated
  $\Z$-subalgebra $A_0 \subseteq A$ together with a finite type
  morphism of algebraic stacks $X_0 \to \spec A_0$ with finite
  diagonal such that $X_0 \times_{\spec A_0} \spec A \simeq X$
  \cite[App.~B]{rydh-2009}. Since $X_0$ has finite diagonal, the
  Keel--Mori Theorem \cite{MR1432041,MR3084720} implies that it admits
  a coarse moduli space $\pi_0 \colon X_0 \to Y_0$, where $Y_0$ is a
  separated algebraic space of finite type over $\spec A_0$. Set
  $Y=Y_0 \times_{\spec A_0} \spec A$ and let $\pi \colon X \to Y$ be
  the induced morphism, which is of finite presentation. Note that
  $X\simeq X_0\times_{Y_0} Y$ as well but $\pi$ is not, in general, a
  coarse moduli space (it is if $X_0$ is tame or
  $\spec A \to \spec A_0$ is flat, however). Let $G\in \DQCOH(X)^c$
  and $P \in \DQCOH(Y)^c= \PERF(Y)$; then
  \begin{align*}
    \RHom_{\Orb_X}(\LDERF \pi^*P^\vee\tensor^{\LDERF}_{\Orb_X} G,M) &\simeq \RHom_{\Orb_X}(\LDERF \pi^*P^\vee,\SRHom_{\Orb_X}(G,M))\\
                                                                    &\simeq \RHom_{\Orb_Y}(P^\vee,\RDERF \pi_{\qcsubscript,*}\SRHom_{\Orb_X}(G,M))\\
                                                                    &\simeq \RDERF\Gamma(Y,P\otimes^{\LDERF}_{\Orb_Y}\RDERF \pi_{\qcsubscript,*}\SRHom_{\Orb_X}(G,M)).
  \end{align*}
  Now $\LDERF\pi^*P^\vee \tensor^{\LDERF}_{\Orb_X} G \in \DQCOH(X)^c$
  \cite[Lem.~4.4(2)]{perfect_complexes_stacks} and so the left hand
  side of the above is pseudocoherent. Hence,
  $\RDERF \pi_{\qcsubscript,*}\SRHom_{\Orb_X}(G,M)$ is pseudocoherent relative to $A$
  for all $G\in \DQCOH(X)^c$ \cite[Tag \spref{0GFJ}]{stacks-project}
  (resp.~pseudocoherent relative to $A$ and bounded \cite[Tag
  \spref{0GFE}]{stacks-project}).
  
  Observe that there is an \'etale covering
  $Y_0' \xrightarrow{y_0} Y_0$, where $Y_0'$ is an affine scheme,
  together with a finite and faithfully flat morphism of finite
  presentation $w_0 \colon W_0' \to X_0\times_{Y_0} Y_0'$, where
  $W_0'$ is an affine scheme (see, for example, \cite[Proof of
  Thm.~6.12]{MR3084720}). Now form the diagram:
  \[
    \xymatrix{ & W' \ar[d]_{w} \ar[rr]  & & W_0' \ar[d]^{w_0} \\& X' \ar[dl]_x \ar[rr]\ar[dd]^(0.3){\pi'}|\hole & & X_0' \ar[dd]^{\pi_0'} \ar[dl]\\X \ar[rr] \ar[dd]_\pi & &X_0 \ar[dd]^(0.3){\pi_0} & \\ & Y' \ar[rr]|\hole \ar[ddl]|\hole \ar[dl]_y& & Y_0' \ar[dl] \ar[ddl]\\
      Y \ar[rr] \ar[d] & &\ar[d] Y_0 &\\
      \spec A \ar[rr] & & \spec A_0, &}
  \]
  where all squares are cartesian. Observe that if $P \in \DQCOH(X)$, then
  \[
    y^*\RDERF \pi_{\qcsubscript,*}\SRHom_{\Orb_X}(P,M) \simeq \RDERF
    \pi'_{\qcsubscript,*}x^*\SRHom_{\Orb_X}(P,M) \simeq \RDERF
    \pi'_{\qcsubscript,*}\SRHom_{\Orb_{X'}}(x^*P,x^*M).
  \]
  It follows from the above that the right hand side is pseudocoherent
  relative to $A$ (resp.~pseudocoherent relative to $A$ and bounded)
  whenever $P\in \DQCOH(X)^c$. Let
  $\mathcal{T} \subseteq \DQCOH(X')^c$ be the full subcategory with
  objects those $G'$ such that the complex
  $\RDERF \pi'_{\qcsubscript,*}\SRHom_{\Orb_{X'}}(G',x^*M)$ is pseudocoherent
  relative to $A$ (resp.~pseudocoherent relative to $A$ and
  bounded). Certainly, $\mathcal{T}$ is a thick and triangulated
  subcategory. The above shows that
  $x^*\DQCOH(X)^c \subseteq \mathcal{T}$. But $x$ is affine, so the
  smallest thick subcategory of $\DQCOH(X')^c$ containing
  $x^*\DQCOH(X)^c$ is $\DQCOH(X')^c$ \cite[Lem.~8.2 \&
  Thm.~3.12]{perfect_complexes_stacks} (essentially Thomason's
  Theorem). It follows that $\mathcal{T} = \DQCOH(X')^c$ and
  $\RDERF \pi'_{\qcsubscript,*}\SRHom_{\Orb_{X'}}(G',x^*M)$ is pseudocoherent
  relative to $A$ (resp.~pseudocoherent relative to $A$ and bounded)
  for all $G'\in\DQCOH(X')^c$.

  Now $w$ is a finite and faithfully flat morphism of finite
  presentation and $W'$ is affine, so $w_{\qcsubscript,*}\Orb_{W'} \in \DQCOH(X')^c$
  \cite[Ex.~3.8 \& Cor.~4.15]{perfect_complexes_stacks}. It follows
  that
  \begin{align*}
    \RDERF \pi'_{\qcsubscript,*}\SRHom_{\Orb_{X'}}((w_{\qcsubscript,*}\Orb_{W'})^\vee,x^*M) &\simeq \RDERF \pi'_{\qcsubscript,*}(w_{\qcsubscript,*}\Orb_{W'}\tensor^{\LDERF}_{\Orb_{X'}}x^*M)\\
                                                              &\simeq \RDERF \pi'_{\qcsubscript,*}\RDERF w_{\qcsubscript,*}w^*x^*M\\
                                                              &\simeq \RDERF (\pi'\circ w)_{\qcsubscript,*}(w^*x^*M)                                                              
  \end{align*}
  is pseudocoherent relative to $A$ (resp.~pseudocoherent relative to
  $A$ and bounded). But $\pi' \circ w \colon \colon W' \to Y'$ is
  finite, so $w^*x^*M$ is pseudocoherent relative to $A$
  (resp.~pseudocoherent relative to $A$ and bounded)
  \cite[\spref{09UK}]{stacks-project}. By \'etale descent, it follows
  that $M$ is pseudocoherent relative to $A$ (resp.~pseudocoherent
  relative to $A$ and bounded).
\end{proof}
We have the following ``saturation'' result for stacks, along the
lines of \cite{MR1996800}.
\begin{corollary}\label{C:saturated}
  Let $X$ be a proper algebraic stack over a field $k$ with finite
  diagonal. If $F \colon (\DQCOH(X)^c)^\opp \to \Vect(k)$ is
  a cohomological functor of finite type, then it is representable by some
  $M \in \DCAT^b_{\COH}(X)$. 
\end{corollary}
\begin{proof}
  By \cite[Thm.~A]{perfect_complexes_stacks}, $\DQCOH(X)$ is compactly
  generated. Hence, \cite[Lem.~2.14]{MR1867248} provides
  $M \in \DQCOH(X)$ with $F(-) \simeq \Hom_{\Orb_X}(-,M)$. Since $F$
  is of finite type, if $G\in \DQCOH(X)^c$, then
  $\sum_{n\in \Z} \dim_k F(G[n]) < \infty$. It follows that
  $\RHom_{\Orb_X}(G,M) \in \DCAT_{\COH}^b(k)$ for all
  $G\in \DQCOH(X)^c$. Theorem \ref{T:boundedness-stacks} implies that
  $M\in \DCAT^b_{\COH}(X)$.
\end{proof}
\section{Coherent duality for proper and tame morphisms}
Let $f \colon X \to S$ be a concentrated morphism of algebraic
stacks. By \cite[Thm.~4.14]{perfect_complexes_stacks}, there is an
adjoint pair
\[
  \RDERF f_{\qcsubscript,*} \colon \DQCOH(X) \leftrightarrows \DQCOH(S) \colon
  f^\times.
\]
Note that $f^\times$ is much simpler to construct than $f^!$, which is
a partial right adjoint to $\RDERF f_{\qcsubscript,*}$ on a different category
\cite{MR4575464}. Also see \cite{MR4479830} for some related results.

If $f$ is a proper morphism of finite tor-dimension between noetherian
schemes, then $f^\times$ admits an explicit description
\cite[\S5]{MR1308405}. Here we extend this result to proper and tame
morphisms of finite tor-dimension of noetherian algebraic stacks,
using the recent developments discussed in \cite{MR4239176}. We also
observe that Theorem \ref{T:boundedness-stacks} shows that $f^\times$
preserves coherence, which will be important for our applications.
\begin{theorem}\label{T:coherent-duality}
  Let $S$ be a noetherian and concentrated algebraic stack. Let
  $f \colon X \to S$ be a proper and tame morphism of algebraic stacks
  of finite tor-dimension.
  \begin{enumerate}
  \item \label{TI:coherent-duality:bc} If  $j \colon S'\to S$ is
    tor-independent of $f$ and concentrated, where $S'$ is
    noetherian, and $\theta \colon X\times_S S' \to X$ and
    $f' \colon X\times_S S' \to S'$ are the projections; then
    $\theta^*f^\times(-) \simeq f'^\times j^*(-)$.
  \item \label{TI:coherent-duality:coherent} The restriction of
    $f^\times$ to $\DCAT^b_{\COH}(S)$ factors through
    $\DCAT^b_{\COH}(X)$ and
  \item \label{TI:coherent-duality:formula}
    $f^\times(-) \simeq f^\times(\Orb_S) \otimes^{\LDERF}_{\Orb_X}
    \LDERF f^*(-)$.
  \end{enumerate}
\end{theorem}
\begin{proof}
  We note once and for all that $\RDERF f_{\qcsubscript,*}$ sends perfects to
  perfects (Corollary \ref{C:perfect-tor-dim}) and so compacts to
  compacts if $S$ is concentrated.
  
  We begin with \eqref{TI:coherent-duality:bc}. If $S$ and $S'$ are
  quasi-affine, then by \cite[Lem.~7.1]{perfect_complexes_stacks} we
  are in the situation of \cite[Defn.~5.4]{MR4239176} and
  \cite[Prop.~6.3]{MR4239176} gives the base change
  result.
  In general, let $p\colon \spec A \to S$ be a smooth covering; then
  $\RDERF p_{\qcsubscript,*}\Orb_{\spec A}$ is a descendable $\Orb_S$-algebra
  \cite[Thm.~7.1]{hall2022remarks}. In particular, the smallest thick
  subcategory of $\DQCOH(S)$ containing the objects $\RDERF p_{\qcsubscript,*}N$,
  $N \in \DQCOH(\spec A)$ is $\DQCOH(S)$. Hence, it suffices to
  establish the base change isomorphism on objects of the form
  $\RDERF p_{\qcsubscript,*}N$. Straightforward functoriality and tor-independent
  base change \cite[Cor.~4.13]{perfect_complexes_stacks}
  considerations show that it suffices to establish base change when
  $S=\spec A$. If $S'$ is quasi-affine, then the case already
  considered establishes base change. In particular, base change is
  proved whenever $j$ is quasi-affine. A standard bootstrapping
  procedure now gives the result whenever $j$ is representable. A
  further bootstrapping gets us the general result.

  For \eqref{TI:coherent-duality:coherent} and \eqref{TI:coherent-duality:formula}, by
  \eqref{TI:coherent-duality:bc} we may reduce to the situation where
  $S=\spec A$ is affine. Let $P\in \DQCOH(X)$ be perfect and
  $M\in \DCAT^b_{\COH}(S)$; then
  \[
    \RHom_{\Orb_X}(P,f^\times({M})) \simeq \RHom_{\Orb_S}(\RDERF f_{\qcsubscript,*}P,M).
  \]
  Since $\RDERF f_{\qcsubscript,*}P$ is perfect,
  $\RHom_{\Orb_S}(\RDERF f_{\qcsubscript,*}P,M) \in \DCAT^b_{\COH}(A)$. Theorem
  \ref{T:boundedness-stacks} now implies that
  $f^\times(M) \in \DCAT^b_{\COH}(X)$, which gives
  \eqref{TI:coherent-duality:coherent}. Finally,
  \eqref{TI:coherent-duality:formula} follows from
  \cite[Lem.~2.4(1)]{MR4239176} and the preservation of perfects/compacts.
\end{proof}
\begin{corollary}\label{C:gorenstein-duality}
  Let $S$ be a noetherian and concentrated algebraic stack. Let
  $f\colon X \to S$ be a proper and tame morphism of algebraic stacks. If $f$ is flat with geometrically Gorenstein fibers of dimension $d$, then
  \begin{enumerate}
  \item\label{CI:gorenstein-duality:perfect} $f^{\times}$ sends perfects to perfects;
  \item\label{CI:gorenstein-duality:lines} if $L$ is a line bundle on
    $S$, then $f^{\times}L[-d]$ is a line bundle on $X$.
  \end{enumerate}
\end{corollary}
\begin{proof}
  By Theorem \ref{T:coherent-duality}\itemref{TI:coherent-duality:bc},
  we may assume that $S$ is affine and $L\simeq \Orb_S$. Since
  $\PERF(S)$ is the smallest thick subcategory containing $\Orb_S$, it
  suffices to prove that $f^{\times}\Orb_S[-d]$ is a line bundle. By
  Theorem
  \ref{T:coherent-duality}\itemref{TI:coherent-duality:coherent}, we
  also know that $f^{\times}\Orb_S[-d] \in \DCAT^b_{\COH}(X)$. By
  Lemma \ref{L:bridgeland-nakayama}, $f^{\times}\Orb_S[-d]$ is a line bundle on $X$ if
  and only if $\LDERF s_X^*f^{\times}\Orb_S[-d]$ is a line bundle
  on $X_{\ell}$ for all closed points $s \colon \spec \ell \to S$. Since
  $f$ is flat, it is tor-independent of all geometric closed points
  $s$. Hence, Theorem
  \ref{T:coherent-duality}\itemref{TI:coherent-duality:bc} and the
  previous assertion show that we are reduced to the situation where
  $S=\spec k$ and $k$ is an algebraically closed field and $X$ is Gorenstein of dimension $d$. Applying Lemma \ref{L:bridgeland-nakayama} again (with $f=\mathrm{id}_X$), we see that it suffices to prove that $\LDERF x^*f^\times k[-d]$ is a line bundle on $\spec k$ for all closed points $x\colon \spec k \to X$. Since $k$ is algebraically
  closed, $x$ factors through a smooth covering $U \to X$, with $U$ a
  Gorenstein and affine scheme of dimension $d$. Let $u\in U$ be the image of
  $x$. Since $\Orb_{U,u}$ is Gorenstein, it is Cohen--Macaulay, and so
  admits a regular sequence $(a_1,\dots, a_d)$ such that
  $B=\Orb_{U,u}/(a_1,\dots,a_d)$ is artinian and Gorenstein. Let
  $\tilde{x} \colon \spec B \to X$ be the induced map. Also let
  $\beta = f\circ \tilde{x}$, which is finite; then $\tilde{x}$ is
  finite and has tor-dimension $d$. By Theorem
  \ref{T:coherent-duality}\itemref{TI:coherent-duality:formula}:
  \[
    \beta^{\times}k \simeq \tilde{x}^\times f^{\times}k \simeq
    (\tilde{x}^\times\Orb_X) \tensor^{\LDERF}_{\Orb_{\spec B}} \LDERF
    \tilde{x}^*f^\times k.
  \]
  Since $B$ is Gorenstein of dimension $0$, $\beta^{\times}k$ is a
  (trivial) line bundle on $\spec B$.  A simple argument shows that
  $\LDERF \tilde{x}^*f^\times k$ must be a shift of a line bundle on
  $\spec B$. For the degree, we note that $\tilde{x}_*$ is $t$-exact
  and
  $\tilde{x}_*\tilde{x}^{\times}\Orb_{X} \simeq
  \SRHom_{\Orb_X}(\tilde{x}_*\Orb_{\spec B},\Orb_X)$, which lives in
  degree $d$ and so $\LDERF \tilde{x}^*f^\times k$ lives in degree
  $-d$; in particular, $\LDERF \tilde{x}^*f^\times k[-d]$ is a line
  bundle. It follows immediately that $\LDERF x^*f^\times k[-d]$ is a
  line bundle.
\end{proof}
The following variant of \cite[Lem.~4.3]{MR1651025} is used in the
proof of Corollary \ref{C:gorenstein-duality}. Note that flatness is
essential (e.g., it is obviously false for
$f\colon X = \spec \C[x]/(x) \to \spec \C[x]=S$ and
$\mathcal{E} = \Orb_X$ and $s=f$).
\begin{lemma}\label{L:bridgeland-nakayama}
  Let $f \colon X \to S$ be a flat morphism of locally noetherian
  algebraic stacks. Let $F \in \DCAT^-_{\COH}(X)$ be such that
  $\LDERF s_X^*F$ is a sheaf on $X_{\ell}$ for some point
  $s\colon \spec \ell \to S$ of $S$, where
  $s_X \colon X_{\ell} \to X$ denotes the projection and $\ell$ is a field. Then there is
  an open neighbourhood $U \subseteq X$ of $s$ such that $F_U$ is a
  coherent sheaf on $U$, flat over $S$.
\end{lemma}
\begin{proof}
  This is all local on $S$ and $X$ and can be further checked after
  localization. Hence, we may assume that $S=\spec A$, where $A$ is
  local with maximal ideal $\mathfrak{m}$, $X=\spec B$ and
  $\ell = A/\mathfrak{m}$. It suffices to prove that if
  $F \tensor^{\LDERF}_B (B \tensor_A A/\mathfrak{m})$ is a module,
  then $F$ is a module that is flat over $A$. Since $B$ is flat over $A$,
  \[
    F \tensor^{\LDERF}_B (B \tensor_A A/\mathfrak{m}) \simeq F
    \tensor^{\LDERF}_B (B \tensor_A^{\LDERF} A/\mathfrak{m}) \simeq F
    \tensor^{\LDERF}_A A/\mathfrak{m}.
  \]
  Choose $u_0\geq 0$ such that $\COHO{u}(F) = 0$ for all $u>u_0$; then
  $0=\COHO{u_0}(F \tensor^{\LDERF}_A A/\mathfrak{m}) \simeq
  \COHO{u_0}(F) \tensor_A A/\mathfrak{m}$. By Nakayama's Lemma,
  $\COHO{u_0}(F) = 0$ unless $u_0 \leq 0$. Next observe that there is a distinguished triangle
  \[
    \trunc{<0}F \to  F \to \COHO{0}(F)[0] \to \trunc{<0}F[1].
  \]
  The cohomology sequence associated to the
  functor: $-\tensor_A^{\LDERF} A/\mathfrak{m}$ gives an exact sequence
  \[
    \COHO{-1}(F \tensor_A^{\LDERF} A/\mathfrak{m}) \to \Tor_1^A(\COHO{0}(F),A/\mathfrak{m}) \to \COHO{0}(\trunc{<0}F \tensor^{\LDERF}_AA/\mathfrak{m}).
  \]
  The outside terms of this clearly vanish, so we have
  $\Tor_1^A(\COHO{0}(F),A/\mathfrak{m}) = 0$. Since $\COHO{0}(F)$ is a
  finitely generated $B$-module, it follows from the local criterion
  for flatness that $\COHO{0}(F)$ is flat over $A$. In particular,
  $F \tensor^{\LDERF}_A A/\mathfrak{m} \simeq \COHO{0}(F)
  \tensor^{\LDERF}_A A/\mathfrak{m}$ and so
  $(\trunc{<0}F) \tensor^{\LDERF}_A A/\mathfrak{m} \simeq 0$. The
  (derived) Nakayama Lemma implies that $\trunc{<0}F = 0$ and we have
  the claim.  
\end{proof}

\section{The stack of coherent sheaves}
Let $f\colon X \to S$ be a morphism of algebraic stacks. If $f$ is a
smooth and separated morphism of algebraic spaces of finite
presentation, then the diagonal
$\Delta_{X/S} \colon X \to X\times_S X$ is a closed immersion and the
resulting morphism $X \to \underline{\mathsf{Hilb}}_{X/S}$ to the
Hilbert functor is an open
immersion.

In general, there is the \emph{stack of coherent sheaves over $S$},
$\underline{\COH}_{X/S} \to S$, which assigns to each $T \to S$ the
groupoid of finitely presented $\Orb_{T\times_S X}$-modules that are
flat over $T$. If $f$ is proper and of finite presentation, then
$\underline{\COH}_{X/S}\to S$ is a morphism of algebraic stacks that
is locally of finite presentation with affine diagonal
\cite[Thm.~8.1]{MR3589351}.
\begin{proposition}\label{P:key}
  Let $f\colon X \to S$ be a proper and smooth morphism of algebraic
  stacks with finite diagonal. Then $(\Delta_{f})_*\Orb_X$, viewed as
  a finitely presented $\Orb_{X\times_S X}$-module that is flat over
  $X$, defines a smooth morphism
  $\delta_f \colon X \to \underline{\COH}_{X/S}$.
\end{proposition}
\begin{proof}
  Since $f$ is of finite presentation with finite diagonal, $\Delta_f$
  is finite and of finite presentation. It follows that
  $(\Delta_{f})_*\Orb_X$ is a finitely presented
  $\Orb_{X\times_S X}$-module that is flat over $X$ via the first projection.

  We next check that the morphism is smooth using the infinitesimal
  lifting criterion. Let $A$ be a strictly henselian local ring and
  $I \subseteq A$ a square zero ideal. Let $A_0 =A/I$; then we must
  complete the diagram:
  \[
    \xymatrix{\spec A_0 \ar[d]_i \ar[r]^{{a}_0} & X \ar[d]^{\delta_f}
      \\\spec A \ar[r]_a \ar@{-->}[ur]& \underline{\COH}_{X/S}.}
  \]
  In other words, we must show that if we are given a finitely
  presented $\Orb_{\spec A \times_S X}$-module $F$ that is flat over
  $\spec A$ together with an isomorphism
  $\lambda_0 \colon F_{\spec A_0 \times_S X} \simeq (a_0)_{X\times_S
    X}^*(\Delta_f)_*\Orb_X$, then $a_0$ lifts to a morphism
  $\tilde{a} \colon \spec A \to X$ and $\lambda_0$ lifts to an
  isomorphism
  $\lambda \colon F \simeq \tilde{a}^*_{X\times_S
    X}(\Delta_f)_*\Orb_X$. Since $X \to S$ is smooth, we can lift the
  composition $\spec A \to \underline{\COH}_{X/S} \to S$ to
  $\tilde{a} \colon \spec A \to X$. Now form the following cartesian
  diagram:
  \[
    \xymatrix@C+2pc{\spec A_0 \ar[d]_{h_0} \ar@/^2pc/[rr]^-{a_0} \ar[r]_{i} & \spec A \ar[d]^{\tilde{h}}\ar[r]_{\tilde{a}} & X \ar[d]^{\Delta_f}\\
      \spec A_0 \times_S X \ar@/_2pc/[rr]_-{(a_0)_{X\times_S X}} \ar[r]^{i_{\spec A \times_S X}} & \spec A \times_S X  \ar[r]^{\tilde{a}_{X\times_S X}} &
      X\times_{S} X,}
  \]
  where $\tilde{h} = (\mathrm{id},\tilde{a})$ and
  $h_0 = (\mathrm{id},a_0)$.  Thus, it suffices to complete the
  following diagram:
  \[
    \xymatrix@C+4pc{ F \ar@{-->}[r]^{\exists\lambda} \ar@{->>}[d] &
      \tilde{a}^*_{X\times_S X}(\Delta_f)_*\Orb_X
      \ar@{->>}[d]\\
      (i_{\spec A \times_S X})_*F_{\spec A_0 \times_S X}
      \ar[r]_-{(i_{\spec A \times_S X})_*\lambda_0} & (i_{\spec A \times_S X
        })_*(a_0)_{X\times_S X}^*(\Delta_f)_*\Orb_X, }
  \]
  where the right vertical arrow is the composition:
  \[
    \tilde{a}^*_{X\times_S X}(\Delta_f)_*\Orb_X \twoheadrightarrow (i_{\spec A \times_S X
      })_*i_{\spec A \times_S X}^*\tilde{a}^*_{X\times_S
      X}(\Delta_f)_{*}\Orb_X \simeq (i_{\spec A \times_S X})_*(a_0)^*_{X\times_S X}(\Delta_f)_{*}\Orb_X
  \]
  By affine base change,
  \begin{align*}
    \tilde{a}^*_{X\times_S X}(\Delta_f)_{*}\Orb_X &\simeq \tilde{h}_{*}\Orb_{\spec A} \quad \mbox{and}\\
    (i_{\spec A \times_S X})_*(a_0)_{X\times_S X}^*(\Delta_f)_{*}\Orb_X &\simeq (i_{\spec A \times_S X})_*(h_0)_{*}\Orb_{\spec A_0} \simeq \tilde{h}_{*}i_{*}\Orb_{\spec A_0} 
  \end{align*}
  Thus, adjunction says that it is sufficient to lift
  \[
    \xymatrix{\tilde{h}^*F \ar@{-->}[r] \ar[dr] &
      \Orb_{\spec A} \ar@{->>}[d]\\ &
      i_{*}\Orb_{\spec A_0}.}
  \]
  Since $\spec A$ is affine, it remains to prove that $\tilde{h}^*F$
  is a vector bundle on $\spec A$. First we note that $\tilde{h}^*F$
  is a finitely presented $\Orb_{\spec A}$-module, so it remains to
  prove that it is flat, which is clear. 
\end{proof}
\section{Fourier--Mukai transforms}\label{S:FM}
In this section we discuss Fourier--Mukai transforms in the relative
situation for algebraic stacks. Much of this is likely folklore,
cf. \cite{MR2669705,MR2505443}. Consider a cartesian diagram of
algebraic stacks:
\[
  \xymatrix{& \ar[dl]_{p} X\times_S Y \ar[dr]^{q} & \\ X \ar[dr]_f & & Y \ar[dl]^g \\ & S. &}
\]
Let $K \in \DQCOH(X\times_S Y)$. Then we may form the \emph{Fourier--Mukai functor} or \emph{integral transform}
\[
  \Phi_{K,X\to Y}(-) = \RDERF q_{\qcsubscript,*} (K
  \otimes^{\LDERF}_{\Orb_{X\times_S Y}} \LDERF p^*(-)) \colon
  \DQCOH(X) \to \DQCOH(Y).
\]
If $h \colon Z \to S$ is another morphism and
$L \in \DQCOH(Y\times_S Z)$, let
$\pi_{12} \colon X\times_S Y \times_S Z \to X\times_S Y$ and similarly
for $\pi_{13}$, $\pi_{23}$ be the associated projections. Let
$r \colon Y\times_S Z \to Y$ and $s \colon Y\times_S Z \to Z$ be the
projections. Assume that $r$ and $q$ are tor-independent and $f$, $g$,
$h$ are concentrated.  Then the projection formula and tor-independent
base change \cite[\S4]{perfect_complexes_stacks} give the well-known
formula:
\begin{align*}
  \Phi_{L,Y \to Z} \circ \Phi_{K,X\to Y}(M)
                                            &\simeq \Phi_{\RDERF \pi_{13,\qcsubscript,*}(\LDERF \pi_{23}^*L \tensor^{\LDERF}_{\Orb_{X\times_S Y\times_S Z}} \LDERF \pi_{12}^*K), X \to Z}(M).
\end{align*}
Hence, the composition of Fourier--Mukai transforms is a Fourier--Mukai transform. We now have remarks about preservation of coherence. 
\begin{remark}\label{R:bounded-above}
  If $X$ and $Y$ are noetherian, $q$ is proper and tame, and
  $K \in \DCAT^-_{\COH}(X\times_S Y)$, then $\Phi_{K,X\to Y}$ sends
  $\DCAT^-_{\COH}(X)$ to $\DCAT^-_{\COH}(Y)$ (Proposition
  \ref{P:necessary}).
\end{remark}
\begin{remark}\label{R:bounded}
  Let $W \to T$ be a flat and locally of finite presentation morphism
  of algebraic stacks. We say that $P \in \DQCOH(W)$ is
  \emph{$T$-perfect} if it is pseudocoherent on $W$ and locally of
  finite tor-dimension over $T$ \cite[Tag
  \spref{0DI0}]{stacks-project}. In particular, if $T$ is the spectrum
  of a field and $W$ is quasicompact, then $P$ is $T$-perfect if and
  only if $P \in \DCAT^b_{\COH}(W)$. Now assume that $X$ and $Y$ are
  noetherian; $q$ is proper, flat and tame; $p$ is flat and of finite
  type; and $K \in \DQCOH(X\times_S Y)$ is $X$-perfect. Then the
  restriction of $\Phi_{K,X\to Y}$ to $\DCAT^b_{\COH}(X)$ factors
  through $\DCAT^b_{\COH}(Y)$ (Proposition \ref{P:necessary}).
\end{remark}
The following lemma will be useful.
\begin{lemma}\label{L:bounded-perfect-smooth}
  Let $W \to T$ be a flat and locally of finite presentation morphism
  of algebraic stacks. Let $T \to S$ be a smooth morphism of algebraic
  stacks. Let $K \in \DQCOH(W)$ be pseudocoherent. Then $K$ is
  $S$-perfect if and only if it is $T$-perfect.
\end{lemma}
\begin{proof}
  We immediately reduce to the situation where $W=\spec C$,
  $T=\spec B$, and $S=\spec A$. If $K$ is $B$-perfect and $M$ is an
  $A$-module, then
  $K \tensor^{\LDERF}_A M \simeq K \tensor^{\LDERF}_B B
  \tensor^{\LDERF}_A M \simeq K \tensor^{\LDERF}_B (B \tensor_A M)$ by
  flatness of $A \to B$. In particular, the $B$-tor-amplitude of $K$
  bounds its $A$-tor-amplitude, which proves one direction. For the
  other direction, we apply \cite[Tag
  \spref{0GHJ}]{stacks-project}. Since $K$ is $A$-perfect, it is
  bounded below, so we just need to show that
  $K \tensor^{\LDERF}_B \kappa(\mathfrak{p})$ is bounded below for all
  primes $\mathfrak{p} \subseteq B$. Thus, let
  $\mathfrak{p} \subseteq B$ be a prime and let
  $\mathfrak{q} \subseteq A$ be its contraction. Then
  \[
    K \tensor^{\LDERF}_B \kappa(\mathfrak{p}) \simeq K
    \tensor^\LDERF_B (B \tensor_A \kappa(\mathfrak{q}))
    \tensor^{\LDERF}_{B \tensor_A \kappa(\mathfrak{q})}
    \kappa(\mathfrak{p}) \simeq (K \tensor^{\LDERF}_A
    \kappa(\mathfrak{q})) \tensor^{\LDERF}_{B \tensor_A
      \kappa(\mathfrak{q})} \kappa(\mathfrak{p}).
  \]
  Since $K$ is $A$-perfect, 
  $K \tensor^{\LDERF}_A \kappa(\mathfrak{q})$ is a bounded
  below complex of $B \tensor_A \kappa(\mathfrak{q})$-modules. But
  $T \to S$ is smooth, $B \tensor_A \kappa(\mathfrak{q})$ is a regular
  ring of finite projective dimension, so tensoring with
  $\kappa(\mathfrak{p})$ will remain bounded below.
\end{proof}
Using Theorem \ref{T:coherent-duality} and Corollary
\ref{C:gorenstein-duality}, we can now compute the adjoints to
Fourier--Mukai functors. These formulas are well-known for schemes.
\begin{proposition}{\par\noindent}\label{P:adjoints-FM}
\begin{enumerate}
\item \label{PI:adjoints-FM:RIGHT} Assume that $Y$ is noetherian and concentrated.  If $q$ is
  proper and tame of finite tor-dimension, then $\Phi_K$ admits a
  right adjoint $H$. If $K$ is perfect, then
  $H\simeq \Phi_{K^\vee \tensor^{\LDERF}_{\Orb_{X\times_S Y}} q^{\times}\Orb_Y}$. 
\item \label{PI:adjoints-FM:LEFT} Assume that $X$ is noetherian and concentrated and $K$ is
  perfect. If $p$ is proper, tame, and flat with geometrically
  Gorenstein fibers, then
  $\Phi_{K^\vee \otimes^{\LDERF}_{\Orb_{X\times_S Y}}
    p^{\times}\Orb_X}$ is left adjoint to $\Phi_K$.
\end{enumerate}
\end{proposition}
We also have the following relative compatibility results.
\begin{lemma}
  Let $M\in \DQCOH(X)$ and $N\in \DQCOH(S)$. If $q$ is concentrated,
  then there is a natural isomorphism
  \[
    F^{f,g}_{M,N,K} \colon \Phi_{K}(M\otimes^{\LDERF}_{\Orb_X} \LDERF
    f^*N) \to \Phi_K(M) \otimes^{\LDERF}_{\Orb_Y} \LDERF g^*N.
  \]
\end{lemma}
Let $t \colon T \to S$ be a morphism of algebraic stacks. Fix
pullbacks $f_T \colon X_T=T\times_S X \to T$, $t_X\colon X_T \to X$,
$g_T \colon Y_T=T\times_S Y \to T$, $t_Y\colon Y_T \to Y$,
$p_T \colon X_T\times_T Y_T \to X_T$,
$q_T \colon X_T\times_T Y_T \to Y_T$, and
$r_T \colon X_T\times_T Y_T \to X\times_S Y$. Set
$K_T= \LDERF r_T^*K\in \DQCOH(X_T\times_T Y_T)$. We have the following simple lemma.
\begin{lemma}\label{L:fm-bc}
  Let $M\in \DQCOH(X)$ and $N\in \DQCOH(S)$.
  \begin{enumerate}
  \item There is a natural isomorphism
    \[
      B^{f,g,t}_{M,K}\colon \LDERF t_Y^*\Phi_K(M) \to
      \Phi_{K_T}(\LDERF t_X^*M).
    \]
  \item The following diagram commutes:
    \[
      \tiny\xymatrix{\LDERF t_Y^* \Phi_K(M \otimes^{\LDERF}_{\Orb_X} \LDERF f^*N) \ar[r]^{\LDERF t_Y^*F^{f,g}_{M,N,K}} \ar[d]_{B^{f,g,t}_{M \otimes^{\LDERF}_{\Orb_X} \LDERF f^*N,K}} & \LDERF t_Y^*(\Phi_K(M) \otimes^{\LDERF}_{\Orb_Y} \LDERF g^*N) \ar[r] & \LDERF t_Y^*\Phi_K(M) \otimes^{\LDERF}_{\Orb_{Y_T}} \LDERF t_Y^*\LDERF g^*N \ar[d]\\ \Phi_{K_T}(\LDERF t_X^*(M \otimes^{\LDERF}_{\Orb_X} \LDERF f^*N)) \ar[d] &  & \LDERF t_Y^*\Phi_K(M) \otimes^{\LDERF}_{\Orb_{Y_T}} \LDERF g_T^*\LDERF t^*N \ar[d]^{B^{f,g,t}_{M,K} \otimes \mathrm{id}} \\
        \Phi_{K_T}(\LDERF t_X^*M \otimes^{\LDERF}_{\Orb_{X_T}} \LDERF
        t_X^*\LDERF f^*N)) \ar[r] & \Phi_{K_T}(\LDERF t_X^*M
        \otimes^{\LDERF}_{\Orb_{X_T}} \LDERF f_T^*\LDERF t^*N)
        \ar[r]_-{F^{f_T,g_T}_{\LDERF t_X^*M,\LDERF t^*N,K_T}} &
        \Phi_{K_T}(\LDERF t_X^*M) \tensor^{\LDERF}_{\Orb_{Y_T}} \LDERF
        g_T^*\LDERF t^*N}
    \]
  \item If $t$ is concentrated and tor-independent of $p$ and
    $F \in \DQCOH(X_T)$, then the adjoint to the composition
    \[
      \LDERF t_X^*\Phi_K(\RDERF (t_X)_{\qcsubscript,*}F)
      \xrightarrow{B^{f,g,t}_{\RDERF (t_X)_{\qcsubscript,*}F,K}}
      \Phi_{K_T}(\LDERF t_X^*\RDERF (t_X)_{\qcsubscript,*}F)
      \xrightarrow{ \Phi_{K_T}(\epsilon_F)} \Phi_{K_T}(F)
    \]
    induces a natural isomorphism
    \[
      A^{f,g,t}_{F,K} \colon \Phi_{K}(\RDERF (t_X)_{\qcsubscript,*}F)
      \to \RDERF (t_X)_{\qcsubscript,*}\Phi_{K_T}(F).
  \]
  \end{enumerate}
\end{lemma}
We have the following variant of \cite[Thm.~2.4]{MR2323539} for stacks.
\begin{proposition}\label{P:ff-everywhere}
  Let $* \in \{b,-\}$. Assume that:
  \begin{itemize}
  \item $S$ is noetherian and concentrated;
  \item $f$, $g$ are proper, tame, and flat; and 
  \item $K$ is perfect.
  \end{itemize}
  If
  $\Phi_{K_{\bar{s}}} \colon \DCAT^*_{\COH}(X_{\bar{s}}) \to
  \DCAT^*_{\COH}(Y_{\bar{s}})$ is fully faithful for all geometric
  closed points $\bar{s}$ of $S$, then
  $\Phi_{K} \colon \DCAT^*_{\COH}(X) \to \DCAT^*_{\COH}(Y)$ is fully
  faithful.
\end{proposition}
\begin{proof}
  By Proposition \ref{P:adjoints-FM}\itemref{PI:adjoints-FM:RIGHT},
  $F=\Phi_K$ has a right adjoint
  $H=\Phi_{K^\vee \tensor^{\LDERF}_{\Orb_{X\times_S Y}} q^{\times}
    \Orb_Y}$. Both $F$ and $H$ send $\DCAT^*_{\COH}$ to
  $\DCAT^*_{\COH}$ (Theorem
  \ref{T:coherent-duality}\itemref{TI:coherent-duality:coherent} and
  Remark \ref{R:bounded-above}). Now just argue as in
  \cite[Thm.~2.4]{MR2323539}.
\end{proof}
It turns out that Proposition \ref{P:ff-everywhere} can be
strengthened using the remarkable results of \cite{MR2964634}. The
proof immediately adapts to our situation, given Proposition
\ref{P:adjoints-FM} and Theorem \ref{T:coherent-duality}.
\begin{theorem}\label{T:anno-log}
  Assume that:
  \begin{itemize}
  \item $S$ and $X$ are noetherian and concentrated; 
  \item $g$ is proper, tame, and flat with geometrically Gorenstein fibers; and 
  \item $K$ is perfect.
  \end{itemize}
  Let
  $Q=\RDERF \pi_{13,\qcsubscript,*}(\LDERF \pi_{12}^*(K^\vee
  \tensor^{\LDERF}_{\Orb_{X\times_S Y}} p^\times\Orb_X)
  \tensor^{\LDERF}_{\Orb_{X\times_S Y\times_S Z}} \LDERF
  \pi_{12}^*K)$. Then
  \[
    \Phi_{Q}\simeq \Phi_{K^\vee \tensor^{\LDERF}_{\Orb_{X\times_S Y}} p^\times\Orb_X}
    \circ \Phi_K \Rightarrow \Phi_{(\Delta_{X/S})_{\qcsubscript,*}\Orb_X}
  \]
  is induced by a $\Orb_{X\times_S X}$-morphism
  $Q \to (\Delta_{X/S})_{\qcsubscript,*}\Orb_X$.
\end{theorem}
\begin{corollary}\label{C:ff-open}
  Let $* \in \{b,-\}$. Assume that:
  \begin{itemize}
  \item $S$ is noetherian and concentrated;
  \item $f$, $g$ are proper, tame, and flat with geometrically Gorenstein fibers; and 
  \item $K$ is perfect. If $*=b$ also assume that $K$ is $X$ and
    $Y$-perfect.
  \end{itemize}
  If
  $\Phi_{K_{\bar{s}}} \colon \DCAT^*_{\COH}(X_{{s}}) \to
  \DCAT^*_{\COH}(Y_{{s}})$ is fully faithful for some 
  point $s$ of $S$, then there is an open $U \subseteq S$
  containing $s$ such that
  $\Phi_{K_U} \colon \DCAT^*_{\COH}(X_U) \to \DCAT^*_{\COH}(Y_U)$ is
  fully faithful.
\end{corollary}
\begin{proof}
  Proposition \ref{P:adjoints-FM}\itemref{PI:adjoints-FM:LEFT} and
  Remarks \ref{R:bounded-above} and \ref{R:bounded} imply that
  $F=\Phi_K$ admits a left adjoint $G$, which is Fourier--Mukai and
  sends $\DCAT^*_{\COH}$ to $\DCAT^*_{\COH}$. By Lemma \ref{L:fm-bc},
  these functors and their adjoints are compatible with base
  change. By \cite[Prop.~1.18]{MR2323539}, $\Phi_{K_T}$ for some
  $t\colon T\to S$ is fully faithful whenever
  $G_T \circ
  \Phi_{K_T} \Rightarrow \mathrm{id}_{\DQCOH(X_T)}$ is an equivalence. By
  Theorem \ref{T:anno-log}, this natural transformation comes from a
  map $Q_T \to (\Delta_{X_T/T})_*\Orb_{X_T}$, where
  $Q \in \DCAT^*_{\COH}(X\times_S X)$. By derived Nakayama, the locus
  where this map is an isomorphism is an open neighbourhood of
  ${s}$ in $S$.
\end{proof}
\section{The Kodaira--Spencer map and Fourier--Mukai transforms}
Let $f \colon X \to S$ be a proper morphism of algebraic stacks that is of finite presentation. Let $W \to S$ be a morphism of algebraic stacks and let $F$ be a finitely presented $\Orb_{W\times_S X}$-module that is flat over $W$. There is an induced morphism $W \xrightarrow{\underline{F}} \underline{\COH}_{X/S}$. In the language of \cite[\S6]{MR3589351}, for each point $w \colon \spec \ell_w \to W$, there is an exact sequence of morphisms of tangent spaces:
\[
   \Def_{\underline{F}}(w,\ell_w) \to \Def_{W/S}(w,\ell_w) \xrightarrow{\mathrm{d}\underline{F}(w)}  \Def_{\underline{\COH_{X/S}}/S}(\underline{F}(w),\ell_w),
\]
where we refer to $\mathrm{d}\underline{F}(w)$ as the
\emph{Kodaira--Spencer} map.
\begin{remark}\label{R:KS}
  Let $k$ be a field. Let $X \xrightarrow{f} Y$ be a finite type
  morphism of algebraic stacks that are locally of finite type over
  $\spec k$. Let $x \colon \spec \ell_x \to X$ be a point of $X$. If
  $f$ is representable, then
  $\Def_{f}(x,\ell_x) \simeq \Hom_{\ell_x}(x^*\Omega_{f},\ell_x)$
  \cite[Ex.~6.1]{MR3589351}. In particular, if $f$ is also radiciel
  (i.e., finite type and universally injective), $X$ is reduced, and
  $k$ is a field of characteristic $0$; then $\Def_{f}(x,\ell_x)=0$
  (and so
  $\mathrm{d}f(x) \colon \Def_{X/k}(x,\ell_x) \to
  \Def_{Y/k}(x,\ell_x)$ is injective) for $x$ in an open and dense
  subset $U$ of $X$. Indeed,
  $f$ admits a factorization as
  $X \twoheadrightarrow V \hookrightarrow Y$, where the first map is
  schematically dominant (so $V$ is reduced) and the second is a
  closed immersion. By generic smoothness, there is a dense open of
  $U$ of $X$ such that $U \to V$ is smooth. Since $f$ is radiciel, it
  follows that $U \to V$ is also \'etale---even an open
  immersion---and so the restriction of $f$ to $U$ is
  unramified. Hence, $(\Omega_{f})_U = 0$ and the claim follows. 
\end{remark}
We will be concerned with the behaviour of the Kodaira--Spencer map
under a Fourier--Mukai transform. The main result of this section is
Proposition \ref{P:ks-functorial}. We begin by recalling some results
from \cite[Tag \spref{0DIS}]{stacks-project} (also see
\cite{MR2177199}), which gives a general interpretation of the tangent
space to moduli stacks of complexes (even if the moduli stack is not
algebraic).

Let $i \colon X \hookrightarrow X'$ be a \emph{locally nilpotent}
closed immersion of algebraic stacks; that is, it is a closed
immersion defined by a quasicoherent sheaf of ideals
$I \subseteq \Orb_{X'}$ and smooth-locally on $X'$ there always exists
an $n$ with $I^n = (0)$. Let $F \in \DQCOH(X)$. Then there is the set
of isomorphism classes of lifts of $F$ over $i$:
\[
  \Lift[F,i] = \{ (F',\phi) \suchthat
  \mbox{$F' \in \DQCOH(X')$ and
    $\phi \colon \LDERF i^* F' \simeq F$}\}/\simeq.
\]
\begin{remark}
  Note that every morphism $(F'_1,\phi_1) \to (F'_2,\phi_2)$ of lifts
  of $F$ to $X'$ is an isomorphism. Indeed, it suffices to prove that
  if $M \in \DQCOH(X')$ and $\LDERF i^*M \simeq 0$, then $M \simeq
  0$. This is reasonably well-known if $X'$ is noetherian or
  $I=\ker(\Orb_{X'} \to \Orb_X)$ is finitely generated (e.g., using
  derived completions) or $M$ is bounded above. In general, this is
  smooth-local on $X'$, so we may assume that $I^n=(0)$ for some
  $n>0$; by induction on $n$ we may assume that $n=2$. Then
  $\LDERF i^*M \simeq 0$ implies via the derived projection formula
  \cite[Cor.~4.12]{perfect_complexes_stacks} that
  $M \otimes^{\LDERF}_{\Orb_{X'}} \RDERF i_{\qcsubscript,*}K \simeq \RDERF i_{\qcsubscript,*}(\LDERF
  i^*M \otimes^{\LDERF}_{\Orb_X} K) \simeq 0$ for all
  $K \in \DQCOH(X)$. Now $I \simeq \RDERF i_{\qcsubscript,*}(i^*I)$ (note that the
  pullback is underived), so tensoring the triangle
  $\RDERF i_{\qcsubscript,*}(i^*I) \to \Orb_{X'} \to \RDERF i_{\qcsubscript,*}\Orb_X \to \RDERF
  i_{\qcsubscript,*}(i^*I)[1]$ with $M$ we see that $M \simeq 0$.
\end{remark}
Now consider a pair of square zero closed immersions of algebraic
stacks:
\[
  \xymatrix{& X \ar@{_(->}[dl]_i \ar@{^(->}[dr]^j &\\
    X'  & & X''.}
\]
If $f \colon X' \to X''$ is a morphism of algebraic stacks that is
compatible with the above diagram, then there is an induced functor
\[
  f^* \colon \Lift[F,j] \to \Lift[F,i] \colon (F',\phi)
  \mapsto (\LDERF f^*F',\LDERF f^*\phi).
\]

Given a lift $(F',\phi)$ of $F$ over $i$, we may tensor the exact
sequence
\[
  \xymatrix{0 \ar[r] & I \ar[r] & \Orb_{X'} \ar[r] &i_{\qcsubscript,*}\Orb_X \ar[r] &
    0}
\]
by $F'$ to obtain the following distinguished triangle in $\DQCOH(X')$:
\[
  \xymatrix{F' \tensor^{\LDERF}_{\Orb_{X'}} I \ar[r] & F' \ar[r] &
    F'\tensor^{\LDERF}_{\Orb_{X'}} \RDERF i_{\qcsubscript,*}\Orb_X \ar[r] & F' \tensor^{\LDERF}_{\Orb_{X'}} I[1].}
\]
Now assume that $i$ is square zero; then $I \simeq \RDERF i_{\qcsubscript,*}(i^*I)$
(note that $i^*$ is underived here). Using the projection formula and
$\phi$, this triangle is naturally isomorphic to:
\[
  \xymatrix{\RDERF i_{\qcsubscript,*}(F \tensor^{\LDERF}_{\Orb_{X}} i^*I) \ar[r] & F'
    \ar[r]^-{\phi^\vee} & \RDERF i_{\qcsubscript,*}F \ar[r]^-{\delta_{(F',\phi)}} &
    \RDERF i_{\qcsubscript,*}(F \tensor^{\LDERF}_{\Orb_{X}} i^*I)[1].}
\]
For naturality properties of the projection formula, see \cite[App.~A]{GAGA_theorems}. 
Given another lift $(F_0,\phi_0)$ of $F$ over $i$, we see that $(F_0,\phi_0) \to (F',\phi)$ if and only if the composition
\[
  F_0 \xrightarrow{\phi_0^\vee} \RDERF i_{\qcsubscript,*}F \xrightarrow{\delta_{(F',\phi)}} \RDERF i_{\qcsubscript,*}(F\otimes^{\LDERF}_{\Orb_X} i^*I)[1]
\]
is the $0$ homomorphism in $\DQCOH(X')$. By $(\LDERF i^*,\RDERF i_{\qcsubscript,*})$
adjunction and the isomorphism $\phi_0 \colon \LDERF i^*F_0 \simeq F$,
we obtain a map
\[
  \tau_{[F,i]}^{(F_0,\phi_0)}(F',\phi') \colon F \to (F \tensor^{\LDERF}_{\Orb_X} i^*I)[1]
\]
in $\DQCOH(X)$. Hence, there is an injective map
\[
  \tau_{[F,i]}^{(F_0,\phi_0)} \colon \Lift[F,i] \to \Ext^1_{\Orb_X}(F,F\tensor^{\LDERF}_{\Orb_X} i^*I),
\]
which is functorial in $i$, $F_0$ and $F$ and sends $(F_0,\phi_0)$ to
$0$. In fact, this map is even surjective, but we will not need
that.

\begin{proposition}\label{P:ks-functorial}
  Assume that $t\colon T \to S$ is a square zero nilpotent closed
  immersion, which is defined by a quasicoherent ideal $I$. Let
  $F\in \DQCOH(X_T)$ and $K\in \DQCOH(X\times_{S} Y)$. If
  $(F_0,\phi_0)$ is a lift of $F$ along $t_X$, then the following
  diagram commutes:
  \[
    \tiny\xymatrix@C-3pc{\Lift[F,t_X] \ar[rr]^{(F',\phi) \mapsto (\Phi_K(F'),\Phi_K(\phi)\circ B^{f,g,t}_{F',K})} \ar[d]_{\tau^{(F_0,\phi_0)}_{[F,t_X]}} & & \Lift[\Phi_{K_T}(F),t_Y] \ar[d]_{\tau^{(\Phi_K(F_0),\Phi_K(\phi_0)\circ B^{f,g,t}_{F_0,K})}_{[\Phi_{K_T}(F),t_Y]}}\\
      \Ext^1_{\Orb_{X_T}}(F,F \otimes_{\Orb_{X_T}}^{\LDERF} t_X^*
      f^*I) \ar[d] & & \Ext^1_{\Orb_{Y_T}}(\Phi_{K_T}(F),\Phi_{K_T}(F)
      \otimes^{\LDERF}_{\Orb_{Y_T}} t_Y^* g^*I) \ar[d] \\
      \Ext^1_{\Orb_{X_T}}(F,F \otimes_{\Orb_{X_T}}^{\LDERF} \LDERF
      f_T^* t^*I) \ar[dr] &
      &
      \Ext^1_{\Orb_{Y_T}}(\Phi_{K_T}(F),\Phi_{K_T}(F)
      \otimes_{\Orb_{Y_T}}^{\LDERF}\LDERF g_T^* t^*I).\\
    & \Ext^1_{\Orb_{Y_T}}(\Phi_{K_T}(F),\Phi_{K_T}(F
      \otimes_{\Orb_{X_T}}^{\LDERF}\LDERF f_T^* t^*I)) \ar[ur] &}
  \]
  Moreover, the vertical maps are injective. 
\end{proposition}
\begin{remark}
  In the situation of Proposition \ref{P:ks-functorial}: if $F$ is
  quasicoherent sheaf that is flat over $T$, then every lift
  $(F',\phi')$ of $F$ along $t_X$ is a quasicoherent sheaf on $X_S$
  that is flat over $S$. This follows from the local criterion for
  flatness. In addition, if $S$ is noetherian and $X$ and $Y$ are
  locally of finite type over $S$ and $F$ is coherent, then the same
  is true of $F'$.
\end{remark}
\section{Generalized points}\label{S:gen-points}
Let $k$ be a field. 
Let $\mathcal{G}$ be a gerbe over $k$; that is, $\mathcal{G}$ is an
algebraic stack with flat diagonal over $\spec k$ such that the fppf
sheafification of the presheaf of isomorphism classes of $\mathcal{G}$
is isomorphic to $\spec k$. Note that all such $\mathcal{G}$ are
smooth and locally noetherian over $\spec k$. If $G$ is a group
algebraic space that is of finite presentation over $\spec k$, then
$BG$ is a noetherian gerbe over $k$. Every noetherian gerbe with affine diagonal has the
resolution property. 
\begin{proposition}\label{P:g-decomp}
  Let $\mathcal{G}$ be a noetherian gerbe over a field $k$. If $\xi$
  is a simple object of the abelian category $\QCOH(\mathcal{G})$,
  then $\xi$ belongs to $\COH(\mathcal{G})$. In particular, there is a
  set $\hat{\mathcal{G}}$ of representatives of isomorphism classes of
  simple objects of $\QCOH(\mathcal{G})$. Moreover, if $M$ is a
  quasicoherent $\Orb_{\mathcal{G}}$-module, then there is a filtered
  union
  \[
    M = \bigcup_{\lambda} M_\lambda,
  \]
  where $M_\lambda \in \COH(\mathcal{G})$ and each $M_\lambda$ admits a finite filtration
  \[
    M_\lambda^0 \subseteq M_\lambda^1 \subseteq \cdots \subseteq M_\lambda^{n_\lambda} = M_{\lambda},
  \]
  where $M_{\lambda}^i/M_{\lambda}^{i-1} \in \hat{\mathcal{G}}$ for
  all $i$ and $\lambda$. Moreover, if $\mathcal{G}$ has linearly
  reductive stabilizer, then every object of
  $\QCOH(\mathcal{G})$ is projective and injective and every
  quasicoherent $\Orb_{\mathcal{G}}$-module $M$ admits a
  decomposition:
  \[
    M \simeq \bigoplus_{\xi\in \hat{\mathcal{G}}} \xi^{\oplus m(M,\xi)}.
  \]
\end{proposition}
\begin{proof}
  Since the algebraic stack $\mathcal{G}$ is noetherian,
  $\xi = \cup_\lambda \xi_\lambda$ \cite{MR1771927}, where the
  $\xi_\lambda$ are coherent $\Orb_{\mathcal{G}}$-submodules of
  $\xi$. But $\xi$ is simple, so $\xi_\lambda = \xi$ for some
  $\lambda$ sufficiently large. The filtration claims are immediate
  from the Jordan--H\"older decomposition. The decomposition claim
  follows from the projectivity and injectivity of objects claim. The
  projectivity and injectivity of objects follows from Corollary
  \ref{C:loc-proj-inj-glob}, as $\mathcal{G}$ admits a flat cover by
  the spectrum of a field.
\end{proof}
Let $X$ be an algebraic stack. Recall that a \emph{point} of $X$ is a
morphism $x\colon \spec k \to X$, where $k$ is a field. There is
always a factorization of $x$ as
\begin{equation}
  \spec k \xrightarrow{p_x} BG_x \xrightarrow{\bar{x}} X,\label{eq:point}
\end{equation}
where $G_x = \Aut_X(x)$, is the automorphism group algebraic space of
$x$, which is locally of finite type over $k$. If
$y\colon \spec l \to X$ is another point of $X$, then we say that
$x\sim y$ if there is a field $K$ and embeddings $k, l \subseteq K$
together with a $2$-morphism between $\spec K \to \spec k \to X$ and
$\spec K \to \spec l \to X$. The set of equivalence classes of points
forms a topological space $|X|$, with open subsets arising the images
of open immersions into $X$ \cite[\S5]{MR1771927} (also see
\cite[App.~A]{bragg2024murphys}).

A point $x\colon \spec k \to X$ is
\emph{affine} if $x$ is affine and \emph{closed} if its image defines
a closed subset of $|X|$. If $X$ is quasiseparated, there is a
further factorization of \eqref{eq:point}:
\begin{equation}
\spec k \xrightarrow{p_x} BG_x \xrightarrow{\tilde{x}} \mathcal{G}_x \xrightarrow{i_x} X,\label{eq:point-gerbe}
\end{equation}
where $i_x$ is a quasi-affine monomorphism
\cite[Thm.~B.2]{MR2774654}. Note that
$\mathcal{G}_x \xrightarrow{i_x} X$ only depends on the isomorphism
class of $x$ and we call $\mathcal{G}_x$ the \emph{residual
  gerbe}. Since $X$ is quasiseparated, the residual gerbe is always
noetherian. Moreover, $p_x$ and $\tilde{x}$ are faithfully flat. We
say that a point $x$ is \emph{strongly affine} if $i_x$ is an affine
morphism.
\begin{remark}
  Clearly, every closed point is strongly affine. More generally, if
  $X$ has affine diagonal and $x$ is a point with reductive
  automorphism group, then $x$ is strongly affine. The
  open point of $[\Aff^2/\mathrm{SL}_2]$ is \emph{not} strongly affine.
\end{remark}
The following generalizes \cite[\S2.2]{MR4280492}, where they worked
with closed points of Deligne--Mumford stacks over algebraically
closed fields.
\begin{definition}
  Let $X$ be a quasiseparated algebraic stack. A \emph{generalized
    point} of $X$ is a pair $(x,\xi)$, where $x$ is a point and $\xi$
  is a simple object of the abelian category
  $\QCOH(\mathcal{G}_x)$. Let
  $\kappa(x,\xi) = (i_x)_{*}\xi \in \QCOH(X)$. Set
  $\Xi(x) = x_{*}\Orb_{\spec k}$.
\end{definition}
\begin{remark}\label{R:support-points}
  Proposition \ref{P:g-decomp} implies that if $(x,\xi)$ is a
  generalized closed point, then $\kappa(x,\xi)$ is a quasicoherent
  $\Orb_X$-module of finite type supported at $x$. Similarly,
  $\Xi(x)_{X-\{x\}} = 0$, so is supported only at $x$.
\end{remark}
\begin{example}\label{E:regular-rep}
  If $x\colon \spec k \to X$ is a point with linearly reductive
  stabilizer (e.g., $X$ is tame or is a Deligne--Mumford stack in
  characteristic $0$), then
  \[
    \Xi(x) = \oplus_{\xi \in \widehat{\mathcal{G}}_x} \kappa(x,\xi)^{\oplus m(x,\xi)},
  \]
  where $m(x,\xi) > 0$. Indeed, if
  $q_x \colon \spec k \to \mathcal{G}_x$ denotes the induced morphism,
  then simplicity implies that
  \begin{align*}
    \Hom_{\Orb_X}(\kappa(x,\zeta),\kappa(x,\zeta))^{m(x,\zeta)}
    &\simeq \Hom_{\Orb_X}(\kappa(x,\zeta),\Xi(x)) \\
    &\simeq \Hom_{\Orb_X}((i_x)_{*}\zeta,(i_x)_{*}(q_x)_{*}\Orb_{\spec k},)\\
    &\simeq \Hom_{\Orb_{\mathcal{G}_x}}(\zeta,(q_x)_{*}\Orb_{\spec k})\\
    &\simeq \Hom_{\Orb_{\spec k}}(q^*\zeta,\Orb_{\spec k})\\
    &\neq 0.
  \end{align*}
  Hence, $m(x,\zeta) > 0$ for all $\zeta \in \widehat{\mathcal{G}}_x$.
\end{example}
\begin{example}\label{E:regular-rep-summand}
  Let $x \colon \spec k \to X$ be a point. Let $k \subseteq l$ be a
  field extension and let
  $x_l \colon \spec l \to \spec k\xrightarrow{x} X$ be the induced
  point. Then $\Xi(x_l)\simeq \Xi(x)^{[l:k]}$. Indeed, if $\lambda \colon \spec l \to \spec k$ is the induced morphism, then $\lambda_*\Orb_{\spec l} = \Orb_{\spec k}^{[l:k]}$ and
  \[
    \Xi(x_l) = (x_l)_*\Orb_{\spec l} \simeq x_*\lambda_*\Orb_{\spec l}
    \simeq x_*\Orb_{\spec k}^{[l:k]} = \Xi(x)^{[l:k]}.
  \]
\end{example}
\begin{example}\label{E:reg-rep-bc}
  Let $X$ be a quasiseparated algebraic stack over a field $k$. Let
  $x\colon \spec l \to X$ be a point such that $k \subseteq l$ is
  finite and separable. Let $l \subseteq \ell$ be a field extension
  that contains a Galois closure $l'$ of $l$. Universal properties produce
  a unique point $\tilde{x} \colon \spec \ell \to X_{\ell}$ from
  $x$. Then $\Xi(x)_{\ell} \simeq \Xi(\tilde{x})^{[l:k]}$. Indeed,
  form the diagram with cartesian squares:
  \[
    \xymatrix{\spec (\ell\tensor_k l) \ar[r]^-{\lambda_l}
      \ar[d]_{\tilde{x}'} & \spec l \ar[d]^x\\ X_{\ell}
      \ar[r]^-{\lambda_{X}} \ar[d]_{f_\ell} & X \ar[d]^f \\ \spec \ell
      \ar[r]_{\lambda} & \spec k.}
  \]
  Since $k \subseteq l$ is finite and separable, there is an
  irreducible and separable polynomial $q(x) \in k[x]$ such that
  $l \simeq k[x]/(q(x))$. Since $k \subseteq l'$ is Galois and $q(x)$
  has a root in $l \subseteq l'$, $q(x)$ splits in $l'$. In
  particular, there is an isomorphism of $l'$-algebras
  \[
    l' \tensor_k l \simeq l'\tensor_k k[x]/(q(x)) \simeq l'[x]/(q(x)) \simeq \prod_{\sigma} l',
  \]
  where the product is over the $[l:k]$ distinct $k$-embeddings
  $l \subseteq l'$. Hence, there is an isomorphism of $\ell$-algebras:
  \[
    \ell \tensor_k l \simeq \ell \tensor_{l'} (l' \tensor_k l) \simeq \prod_\sigma \ell.
  \]
  The universal property shows that the points
  $\spec \ell \xrightarrow{e_\sigma^\sharp} \spec (\prod_\sigma \ell)
  \to X_{\ell}$ induced by each projection
  $e_\sigma \colon \prod_\sigma \ell \to \ell$ are isomorphic to
  $\tilde{x}$. It follows that
  \begin{align*}
    \Xi(x)_{\ell} &= \lambda_X^*x_*\Orb_{\spec l} \simeq \tilde{x}'_*\Orb_{\spec (\ell \tensor_k l)} \simeq  \tilde{x}'_*(\prod_\sigma (e_\sigma^{\sharp})_*\Orb_{\spec \ell}) \\
                  &\simeq \prod_\sigma \tilde{x}'_*(e_\sigma)^\sharp_*\Orb_{\spec \ell} \simeq \tilde{x}_*\Orb_{\spec \ell}^{[l:k]} = \Xi(\tilde{x})^{[l:k]}.
  \end{align*}
\end{example}
One thing that we like points for is to prove ``Nakayama Lemma'' type results. We
have the following result for modules.
\begin{corollary}\label{C:nakayama}
  Let $X$ be a quasiseparated algebraic stack. Let
  $x \colon \spec k \to X$ be a point. Let
  $F$ be a quasicoherent $\Orb_X$-module. Consider the following
  conditions:
  \begin{enumerate}
  \item \label{CI:nakayama:open} there exists an open
    $U \subseteq X$ containing the image of $x$ such that
    $F_U \simeq 0$;
  \item \label{CI:nakayama:point} $x^*F \simeq 0$;
  \item \label{CI:nakayama:gerbe} $i_x^*F \simeq 0$;
  \item \label{CI:nakayama:gen-points}
    $\Hom_{\Orb_X}(F,\kappa(x,\xi))\simeq 0$ for all
    $\xi \in \widehat{\mathcal{G}_x}$.
  \end{enumerate}
  Then
  \itemref{CI:nakayama:open}$\Rightarrow$\itemref{CI:nakayama:point}$\Leftrightarrow$\itemref{CI:nakayama:gerbe}$\Rightarrow$\itemref{CI:nakayama:gen-points}. If
  $F$ is of finite type, then
  \itemref{CI:vanishing-locus:open}$\Leftarrow$\itemref{CI:vanishing-locus:point}. If
  $F$ is of finite type or $\Aut_X(x)$ is linearly reductive, then
  \itemref{CI:vanishing-locus:gerbe}$\Leftarrow$\itemref{CI:vanishing-locus:gen-points}.
\end{corollary}
\begin{proof}
  The implication
  \itemref{CI:nakayama:open}$\Rightarrow$\itemref{CI:nakayama:point}
  is trivial. As is
  \itemref{CI:nakayama:point}$\Leftrightarrow$\itemref{CI:nakayama:gerbe},
  because the map $\spec k \to BG_x \to \mathcal{G}_x$ is faithfully
  flat. For
  \itemref{CI:nakayama:gerbe}$\Rightarrow$\itemref{CI:nakayama:gen-points},
  we have:
  \[
    \Hom_{\Orb_X}(F,\kappa(x,\xi)) = \Hom_{\Orb_X}(F,(i_x)_{\qcsubscript,*}\xi) \simeq
    \Hom_{\Orb_{\mathcal{G}_x}}(i_x^*F,\xi) =0.
  \]
  The claim
  \itemref{CI:vanishing-locus:open}$\Leftarrow$\itemref{CI:vanishing-locus:point}
  is just the usual Nakayama Lemma. For
  \itemref{CI:vanishing-locus:gerbe}$\Leftarrow$\itemref{CI:vanishing-locus:gen-points}:
  it suffices to prove that $\Hom_{\Orb_{\mathcal{G}_x}}(H,\xi)=0$ for
  all $\xi\in \widehat{\mathcal{G}_x}$ implies that $H=0$ when $H$ is
  of finite type or $\Aut_X(x)$ is linearly reductive. In the former
  case, it follows immediately from the Jordan--Holder decomposition
  of Proposition \ref{P:g-decomp}. In the latter, we note that if
  $H\neq 0$, then there exists some $\xi \in \widehat{\mathcal{G}_x}$
  such that $\xi\subseteq H$. Since $\Aut_X(x)$ is linearly reductive,
  the latter claim follows also follows from Proposition
  \ref{P:g-decomp}.
\end{proof}
\begin{example}
  Let $k$ be a field of characteristic $0$. Let
  $X=B\mathbf{G}_{a,k}$. Let $\pi \colon \spec k \to X$ be the usual
  covering. Then $\Hom_{\Orb_X}(\pi_{\qcsubscript,*}\Orb_{\spec
    k},\Orb_X)=0$. Indeed, a non-zero map
  $\pi_{\qcsubscript,*}\Orb_{\spec k} \to \Orb_X$ must be surjective. Let $K$ be its
  kernel; then since $X$ has cohomological dimension $\leq 1$
  ($B\Ga \simeq [\Aff^2-0/\mathrm{SL}_2]$), there is an exact
  cohomology sequence:
  \[
    \scriptstyle{\shfcoho^0(X,K) \to \shfcoho^0(X,\pi_{\qcsubscript,*}\Orb_{\spec k}) \to \shfcoho^0(X,\Orb_X) \to \shfcoho^1(X,K) \to \shfcoho^1(X,\pi_{\qcsubscript,*}\Orb_{\spec k}) \to \shfcoho^1(X,\Orb_X) \to 0.}
  \]
  But $\shfcoho^1(X,\pi_{\qcsubscript,*}\Orb_{\spec k}) = 0$ and $\shfcoho^1(X,\Orb_X) \neq 0$,
  which is
  impossible. 
\end{example}
\begin{remark}\label{R:simple-kappa}
  If $(x,\xi)$ is a generalized closed point of $X$, then
  $\kappa(x,\xi)$ is a simple object of $\QCOH(X)$. Indeed, let
  $\kappa(x,\xi) \twoheadrightarrow Q$ be a quotient; then
  $Q\simeq (i_x)_{*}Q'$ for some coherent $\Orb_{\mathcal{G}_x}$-module
  $Q'$. Since $(i_x)_{*}$ is fully faithful and exact, $Q'$ must be a
  quotient of $\xi$ and by simplicity must be a trivial quotient. In
  particular, $Q$ is a trivial quotient of $\kappa(x,\xi)$ and so
  $\kappa(x,\xi)$ is simple.
\end{remark}
Recall that a complex $F\in \DQCOH(X)$ is \emph{$m$-pseudocoherent} if
smooth-locally on $X$, there is a morphism $\phi \colon P \to F$,
where $P$ is strictly perfect (i.e., a bounded complex of finitely
generated modules that are direct summands of frees), such that
$\COHO{i}(\phi)$ is an isomorphism for all $i>m$ and is surjective for
$i=m$. In other words, if we form the distinguished triangle:
\[
  P \xrightarrow{\phi} F \to K \to P[1],
\]
then $\trunc{\geq m}K \simeq 0$.
\begin{remark}\label{R:limit-pcoh}
  Let $A$ be a ring and $F$ an $A$-module. Recall that $F$ is finitely
  presented (resp.~generated) as an $A$-module if and only for every
  filtered system (resp.~filtered system with injective transition
  maps) of $A$-modules $\{N_\lambda\}_{\lambda}$, the induced map
  \[
    \varinjlim_\lambda \Hom_A(F,N_\lambda) \to \Hom_A(F,\varinjlim_\lambda N_\lambda)
  \]
  is bijective. The necessity of the condition is straightforward to
  establish using a presentation for $F$. The sufficiency is obtained
  by writing $F$ is a suitable filtered colimit of finitely presented
  or finitely generated $A$-modules. This characterization extends to
  finitely presented and finitely generated $\Orb_X$-modules on a
  quasicompact and quasiseparated algebraic stack
  \cite{rydh-2014,rydh_absolute_approximation}.
  There is a generalization of this for pseudocoherent complexes. Let
  $\{N_\lambda\}$ be a directed system of quasicoherent
  $\Orb_X$-modules and set $N=\varinjlim_\lambda N_\lambda$, which is
  also quasicoherent. Assume that $F$ is $m$-pseudocoherent on $X$. If
  $i<-m$ or $i=-m$ and the transition maps
  $N_\lambda \to N_{\lambda'}$ are injective, then
  \[
    \varinjlim_\lambda \Hom_{\Orb_X}(F,N_\lambda[i]) \simeq
    \Hom_{\Orb_X}(F,N[i]).
  \]
  Note that this recovers the statement for modules as finitely
  generated $\Orb_X$-modules are $0$-pseudocoherent and finitely
  presented $\Orb_X$-modules are $-1$-pseudocoherent.

  To show this, we begin by observing that there is a morphism of
  spectral sequences starting at $E_2$:
  \[
    \xymatrix{ \varinjlim_\lambda \shfcoho^p(X,\SHom_{\Orb_X}(F,N_\lambda[q])) \ar@{=>}[r]\ar[d] &\ar[d]
      \varinjlim_\lambda \Hom_{\Orb_X}(F,N_\lambda[p+q])\\
      \shfcoho^p(X,\SHom_{\Orb_X}(F,N[q])) \ar@{=>}[r] &
\Hom_{\Orb_X}(F,N[p+q]),}
  \]
  so if the vertical morphism on the left is an isomorphism for all
  $(p,q)$ with $i=p+q\leq -m$, then the vertical morphism on the right is
  an isomorphism for $i\leq -m$. For each fixed $q$,
  $\SHom_{\Orb_X}(F,N_\lambda[q])$ is uniformly bounded below and as
  $\shfcoho^p(X,-)$ preserves filtered colimits of such complexes of
  $\Orb_X$-modules \cite[Tag \spref{0739}]{stacks-project} and
  $\shfcoho^{p}(X,-)=0$ if $p<0$, it remains to prove that
  $\SHom_{\Orb_X}(F,N[q]) \simeq \varinjlim_\lambda
  \SHom_{\Orb_X}(F,N_\lambda[q])$ when $q<-m$ or when $q=-m$ and the
  transition maps $N_\lambda \to N_{\lambda'}$ are injective. These
  are local on $X$, so we may assume that $X$ is affine and there is a
  morphism $\phi \colon P \to F$, where $P$ is strictly perfect and
  the cocone of $\phi$, $K$, satisfies $\trunc{\geq m}K \simeq 0$. In
  particular, if $M$ is quasicoherent and $q\leq -m$, then
  $\SHom_{\Orb_X}(K,M[q])\simeq 0$. Hence, from the distinguished
  triangle:
  \[
    \SRHom_{\Orb_X}(P,M) \leftarrow \SRHom_{\Orb_X}(F,M) \leftarrow \SRHom_{\Orb_X}(K,M) \leftarrow \SRHom_{\Orb_X}(P,M)[-1]
  \]
  we obtain that the map
  \[
    \SHom_{\Orb_X}(F,M[q]) \to \SHom_{\Orb_X}(P,M[q]) 
  \]
  is an isomorphism if $q<-m$ and is injective if $q=-m$. Since $P$ is perfect, it follows that
  \[
    \SHom_{\Orb_X}(P,M[q]) \simeq  \COHO{0}(\SRHom_{\Orb_X}(P,\Orb_X) \tensor^{\LDERF}_{\Orb_X} M[q]),
  \]
  which will certainly commute with filtered colimits for all
  $q\in \Z$. Hence, it remains to analyse the $q=-m$ case. We have a
  commuting square:
  \[
    \xymatrix{\varinjlim_\lambda \SHom_{\Orb_X}(F,N_\lambda[-m])
      \ar@{^(->}[d] \ar[r] & \ar@{^(->}[d] \SHom_{\Orb_X}(F,N[-m])\\
      \varinjlim_\lambda \SHom_{\Orb_X}(P,N_\lambda[-m]) \ar[r]^-\simeq
      & \SHom_{\Orb_X}(P,N[-m]). }
  \]
  In particular, the map along the top is injective and it remains to
  establish surjectivity when the transition maps
  $N_\lambda \to N_{\lambda'}$ are injective. The bottom row shows
  that given a fixed $F \to N[-m]$, the composition
  $P \to F \to N[-m]$ factors through some $N_\lambda[-m] \to
  N[-m]$. Since the $N_\lambda \to N_{\lambda'}$ are injective,
  $N_\lambda \to N$ is injective. Hence, there is a morphism of
  distinguished triangles:
  \[
    \xymatrix{P \ar[r] \ar[d] & F \ar[d] \ar[r] & K \ar@{-->}[d]_{\exists} \ar[r] & P[1] \ar[d] \\
      N_\lambda[-m] \ar[r] & N[-m] \ar[r] & (N/N_\lambda)[-m] \ar[r] &
      N_\lambda[-m+1].}
  \]
  But $\trunc{\geq m}K \simeq 0$ and so $K \to (N/N_\lambda)[-m]$ must
  be the $0$ map and $F \to N[-m]$ factors through
  $N_\lambda[-m] \to N[-m]$.
\end{remark}
\begin{corollary}\label{C:vanishing}
  Let $X$ be a quasiseparated algebraic stack. Let
  $x \colon \spec k \to X$ be a strongly affine point. Let
  $F$ be an $m$-pseudocoherent complex of $\Orb_X$-modules. The
  following conditions are equivalent:
  \begin{enumerate}
  \item \label{CI:vanishing:gen}$\Hom_{\Orb_X}(F,\kappa(x,\xi)[i]) =0$ for all $i\leq -m$ and
    $\xi\in \hat{\mathcal{G}}_x$; and
  \item \label{CI:vanishing:push}$\Hom_{\Orb_X}(F,(i_{x})_{\qcsubscript,*}N[i]) = 0$ for all $i\leq -m$ and
    $N\in \QCOH(\mathcal{G}_x)$.
  \end{enumerate}
  If $x$ is closed and $X-\{x\}$ is quasicompact, then these
  conditions are equivalent to:
  \begin{enumerate}[resume]
  \item \label{CI:vanishing:van} $\Hom_{\Orb_X}(F,N[i]) = 0$ for all
    $i\leq -m$ and $N\in \QCOH(X)$ with $N_{X-\{x\}} = 0$.
  \end{enumerate}
\end{corollary}
\begin{proof}
  Certainly, we have
  \itemref{CI:vanishing:van}$\Rightarrow$\itemref{CI:vanishing:push}$\Rightarrow$\itemref{CI:vanishing:gen}. For
  \itemref{CI:vanishing:van}$\Leftarrow$\itemref{CI:vanishing:push},
  since $X$ is quasicompact, we may write $N=\cup_\lambda N_\lambda$,
  where the $N_\lambda$ are finitely generated $\Orb_X$-submodules
  \cite{rydh-2014}. Then
  $(N_\lambda)_{X-\{x\}}\subseteq N_{X-\{x\}} = 0$. Also, if
  $i\leq -m$, Remark \ref{R:limit-pcoh} implies that:
  \begin{align*}
    \Hom_{\Orb_X}(F,N[i]) \simeq \varinjlim_\lambda \Hom_{\Orb_X}(F,N_\lambda[i]).
  \end{align*}
  Hence, we may assume that $N$ is a finite type
  $\Orb_X$-module. Since $X-\{x\} \subseteq X$ is quasicompact, there
  is a finitely generated ideal $J \subseteq \Orb_X$ such that
  $|V(J)| =|\{x\}|$ \cite{rydh-2014}. Since $V(J) \subseteq X$ is a
  finitely presented closed immersion, \cite[Lem.~2.5]{mayer-vietoris}
  implies that there is an $n$ with $I^{n+1}N = 0$. From the exact sequences for each $k$,
  \[
    \Hom_{\Orb_X}(F,I^{k+1}N[i]) \to \Hom_{\Orb_X}(F,I^kN[i]) \to \Hom_{\Orb_X}(F,(I^kN/I^{k+1}N)[i])
  \]
  and descending induction (starting at $k=n$), we see that it is
  sufficient to prove the claim when $IN = 0$. That is,
  $N \simeq (i_x)_{\qcsubscript,*}i_x^*N$, which gives the claim. For
  \itemref{CI:vanishing:push}$\Leftarrow$\itemref{CI:vanishing:gen},
  since $m$-pseudocoherence is preserved by derived pullback,
  adjunction implies that we may assume that $X=\mathcal{G}$, where
  $\mathcal{G}$ is a noetherian gerbe over a field. Proposition
  \ref{P:g-decomp} says that there is a filtered union
  $N=\varinjlim_\lambda N_\lambda$, where each $N_\lambda \in \COH(X)$
  and each $N_\lambda$ admits a finite filtration whose graded pieces
  are of the form $\kappa(x,\xi)$. Again by Remark \ref{R:limit-pcoh},
  we may thus assume that $N \in \COH(X)$ and has the claimed
  filtration. By induction on the length of the filtration, the claim
  follows.
\end{proof}
We have the following derived variant of Corollary \ref{C:nakayama}.
\begin{corollary}\label{C:vanishing-locus}
  Let $X$ be a quasiseparated algebraic stack. Let
  $x \colon \spec k \to X$ be a strongly affine point. Let
  $m\in \Z$. Let $F \in \DQCOH(X)$. Consider the following conditions:
  \begin{enumerate}
  \item \label{CI:vanishing-locus:open} there exists an open
    $U \subseteq X$ containing the image of $x$ such that
    $\trunc{\geq m}F_U \simeq 0$;
  \item \label{CI:vanishing-locus:point}
    $\trunc{\geq m}\LDERF x^*F \simeq 0$;
  \item \label{CI:vanishing-locus:gerbe}
    $\trunc{\geq m}\LDERF i_x^*F \simeq 0$;
  \item \label{CI:vanishing-locus:gen-points}
    $\Hom_{\Orb_X}(F,\kappa(x,\xi)[i])\simeq 0$ for all $i\leq -m$ and
    $\xi \in \widehat{\mathcal{G}}_x$.
  \end{enumerate}
  Then
  \itemref{CI:vanishing-locus:open}$\Rightarrow$\itemref{CI:vanishing-locus:point}$\Leftrightarrow$\itemref{CI:vanishing-locus:gerbe}$\Rightarrow$\itemref{CI:vanishing-locus:gen-points}. If
  $F$ is $m$-pseudocoherent, then
  \itemref{CI:vanishing-locus:open}$\Leftarrow$\itemref{CI:vanishing-locus:point}. If
  $F$ is $m$-pseudocoherent or $\Aut_X(x)$ is linearly reductive, then
  \itemref{CI:vanishing-locus:gerbe}$\Leftarrow$\itemref{CI:vanishing-locus:gen-points}.
\end{corollary}
\begin{proof}
  The implication
  \itemref{CI:vanishing-locus:open}$\Rightarrow$\itemref{CI:vanishing-locus:point}
  is trivial. Also,
  \itemref{CI:vanishing-locus:point}$\Leftrightarrow$\itemref{CI:vanishing-locus:gerbe}
  is trivial from the faithful flatness of
  $\spec k \to \mathcal{G}_x$. For
  \itemref{CI:vanishing-locus:gerbe}$\Rightarrow$\itemref{CI:vanishing-locus:gen-points},
  we have if $i\leq -m$ that
  \[
    \Hom_{\Orb_X}(F,\kappa(x,\xi)[i]) = \Hom_{\Orb_{\mathcal{G}_x}}(\LDERF i_x^*F,\xi[i]) \simeq \Hom_{\Orb_{\mathcal{G}_x}}(\trunc{\geq m}\LDERF i_x^*F,\xi[i]) = 0.
  \]
  For
  \itemref{CI:vanishing-locus:open}$\Leftarrow$\itemref{CI:vanishing-locus:point},
  choose $n_0$ such that $\COHO{n}(F) = 0$ for all $n>n_0$. If
  $n_0\geq m$, then $x^*\COHO{n_0}(F) = \COHO{n_0}(\LDERF x^*F) =
  0$. By Corollary \ref{C:nakayama}, it follows that there is an open
  $U$ of $x$ such that $\COHO{n_0}(F_U) = 0$. Thus, we may replace $X$
  by $U$ and $n_0$ by $n_0-1$. Repeating this until $n_0 = m$ we
  obtain the claim. For
  \itemref{CI:vanishing-locus:gerbe}$\Leftarrow$\itemref{CI:vanishing-locus:gen-points},
  choose $n_0\in \Z$ such that $\COHO{n}(\LDERF i_x^*F) = 0$ for all
  $n>n_0$. If $n_0\geq m$, then
  \[
    0=
    \Hom_{\Orb_X}(F,\kappa(x,\xi)[-n_0])=\Hom_{\Orb_{\mathcal{G}_x}}(\LDERF
    i_x^*F,\xi[-n_0]) = \Hom_{\Orb_{\mathcal{G_x}}}(\COHO{n_0}(\LDERF
    i_x^*F),\xi).
  \]
  By Corollary \ref{C:nakayama}, it follows that
  $\COHO{n_0}(\LDERF i_x^*F) = 0$ and so we may replace $n_0$ by
  $n_0-1$. By induction, it follows that
  $\trunc{\geq m}(\LDERF i_x^*F) \simeq 0$.
\end{proof}
A quasicompact and quasiseparated algebraic stack $X$ is said to have
the \emph{Thomason condition} \cite{perfect_complexes_stacks}, if
$\DQCOH(X)$ is compactly generated and for every quasicompact open
subset $U \subseteq X$, there is a compact perfect complex supported
on the complement.
\begin{example}
  If $X$ is a regular, noetherian algebraic stack with quasi-affine
  diagonal of finite cohomological and Krull dimension, then $X$
  satisfies the Thomason condition
  \cite[Thm.~2.1]{hall2022remarks}. It also holds for quasicompact
  algebraic stacks with quasi-finite and separated diagonal
  \cite[Thm.~A]{perfect_complexes_stacks}; quasicompact and
  quasiseparated Deligne--Mumford stacks in characteristic $0$
  \cite[Thm.~7.4]{etale_dev_add}; and most quasicompact algebraic
  stacks admitting good moduli spaces
  \cite[Prop.~6.14]{etale_local_stacks}.
\end{example}
Let $\mathcal{T}$ be a triangulated category. Let $\Omega$ be a
collection of objects of $\mathcal{T}$. Recall that $\Omega$ is a
\emph{spanning class} \cite[Defn.\ 2.1]{MR1651025} if the following
two conditions hold:
\begin{enumerate}
\item if $t\in \mathcal{T}$ and  $\Hom_{\mathcal{T}}(t,\omega[n]) = 0$ for all $\omega\in \Omega$ and $n\in \Z$, then $t\simeq 0$;
\item if $t\in \mathcal{T}$ and $\Hom_{\mathcal{T}}(\omega[n],t) = 0$
  for all $\omega\in \Omega$ and $n\in \Z$, then $t\simeq 0$.
\end{enumerate}
The following result generalizes \cite[Prop.~2.1]{MR4280492} from
smooth Deligne--Mumford stacks to all noetherian algebraic stacks
satisfying the Thomason condition. Note that this is new, even for
noetherian schemes, where it was previously only known in the
Gorenstein case \cite[Ex.~2.2]{MR1651025}.
\begin{corollary}\label{C:spanning}
  Let $X$ be a noetherian algebraic stack. Set
  \[
    \Omega = \{ \kappa(x,\xi) \colon \mbox{$(x,\xi)$ is a generalized
      closed point of $X$} \}.
  \]
  If $X$ satisfies the Thomason condition, then $\Omega$ is a spanning
  class for $\DCAT^b_{\COH}(X)$.
\end{corollary}
\begin{proof}
  Let $F \in \DCAT^b_{\COH}(X)$. Then $F$ is pseudocoherent; so if
  $\Hom_{\Orb_X}(F,\kappa(x,\xi)[i]) = 0$ for all $i\in \Z$ and
  generalized closed points $(x,\xi)$ of $X$, then
  $\trunc{\geq m}F \simeq 0$ for all $m\in \Z$ (Corollary
  \ref{C:vanishing-locus}). Hence, $F\simeq 0$.

  Now assume that $\Hom_{\Orb_X}(\kappa(x,\xi)[i],F) = 0$ for all
  $i\in \Z$ and generalized closed points $(x,\xi)$ of $X$. For each
  closed point $x$ of $X$, let $U_x = X-\{x\}$. Let
  $j_x \colon U_x \subseteq X$ be the induced quasicompact open
  immersion. Let $P \in \DCAT^b_{\COH}(X)$ satisfy $j_x^*P \simeq 0$;
  then we claim that $\Hom_{\Orb_X}(P,F) = 0$. By induction and
  shifting, we may assume that $P\simeq \tilde{P}[0]$ for some
  $\tilde{P} \in \COH(X)$ with $j_x^*P = 0$. If
  $\mathfrak{m}_x \subseteq \Orb_X$ denotes the coherent ideal sheaf
  defining $\{x\}$, then there is a $r$ such that
  $\mathfrak{m}_x^{r+1}\tilde{P} \simeq 0$. Then $\tilde{P}$ admits a
  finite filtration by objects of the form $(i_x)_{\qcsubscript,*}P'$, where
  $P' \in \COH(\mathcal{G}_x)$; hence, we are further reduced to this
  situation. The claim now follows from Proposition \ref{P:g-decomp}.
  Let $y$ be a closed point of $X$. Since $X$ satisfies the Thomason
  condition, there is a set of perfect complexes
  $\{Q_y^i\}_{i\in I_y}$ such that $j_y^*Q_y^i = 0$ for all $i$ that
  compactly generate $\DCAT_{\qcsubscript,\{y\}}(X)$. Form a
  distinguished triangle
  \[
    F_y \to F \to \RDERF j_{y,*} j_y^*F \to F_y[1],
  \]
  where $F_y \in \DCAT_{\qcsubscript,\{y\}}(X)$. Now apply
  $\RHom_{\Orb_X}(Q_y^i,-)$ to obtain a distinguished triangle
  \[\scriptstyle{
    \RHom_{\Orb_X}(Q_y^i,F_y) \to \RHom_{\Orb_X}(Q_y^i,F) \to
    \RHom_{\Orb_X}(Q_y^i, \RDERF j_{y,*} j_y^*F) \to
    \RHom_{\Orb_X}(Q_y^i,F_y)[1].}
  \]
  for all $i\in I_y$. Then $\RHom_{\Orb_X}(Q_y^i,F)\simeq 0$ by the
  previous paragraph and
  \[
    \RHom_{\Orb_X}(Q_y^i, \RDERF j_{y,*} j_y^*F) \simeq
    \RHom_{\Orb_{U_y}}(j_y^*Q_y^i, j_y^*F) \simeq 0.
  \]
  Hence, $\RHom_{\Orb_X}(Q_y^i,F_y) = 0$ for all $i\in I_y$; in
  particular, $F_y = 0$ and $F \simeq \RDERF j_{y,*}j_y^*F$.  But
  tor-independent base change
  \cite[Cor.~4.13]{perfect_complexes_stacks} gives:
  \[
    \LDERF y^*F \simeq \LDERF y^*\RDERF j_{y,*}j_y^*F = 0. 
  \]
  because $y^{-1}(X-\{y\}) = \emptyset$. The result now follows from
  Corollary \ref{C:vanishing-locus}.
\end{proof}
We also have the following version of \cite[Lem.~5.2]{MR1651025} and
\cite[Lem.~2.2]{MR4280492} for algebraic stacks.
\begin{corollary}\label{C:is-sheaf}
  Let $X$ be a quasicompact and quasiseparated algebraic stack. Let
  $x \colon \spec k \to X$ be a closed point of $X$. Let
  $F\in \DQCOH^b(X)$ be pseudocoherent.
  \begin{enumerate}
  \item \label{CI:is-sheaf:support} Assume that
    \[
      \Hom_{\Orb_X}(F,\kappa(y,\zeta)[i]) = 0 \quad \mbox{for all
        $i \in \Z$ and $\zeta\in \widehat{\mathcal{G}_y}$ unless $y\sim x$,}
    \]
    where $(y,\zeta)$ is a generalized closed point. Then
    $F_{X-\{x\}} \simeq 0$.
  \item \label{CI:is-sheaf:bounded-above} Assume \itemref{CI:is-sheaf:support} and 
    \[
      \Hom_{\Orb_X}(F,\kappa(x,\zeta)[i]) =0 \quad \mbox{for all $i<0$
        and $\zeta\in \widehat{\mathcal{G}_x}$}.
    \]
    Then $\trunc{>0}F \simeq 0$.
  \item \label{CI:is-sheaf:bounded-below} Assume
    \itemref{CI:is-sheaf:support}, $\Orb_X$ is coherent and that there
    is a $d\geq 0$ such that
    \[
      \Hom_{\Orb_X}(F,\kappa(x,\zeta)[i]) =0 \quad \mbox{for all $i>d$
        and $\zeta\in \widehat{\mathcal{G}_x}$}.
    \]
    Then $\trunc{<d}F \simeq 0$.
  \item \label{CI:is-sheaf:regular} Assume
    \itemref{CI:is-sheaf:bounded-above} and
    \itemref{CI:is-sheaf:bounded-below}, $X$ has affine diagonal, and
    $x$ admits a factorization
    $\spec k \xrightarrow{j} \spec A \xrightarrow{\hat{x}} X$, where
    $\hat{x}$ is flat, $j$ is a closed immersion, and $A$ is a regular
    local ring of dimension $d$. 
    Then $F$ is a coherent \emph{sheaf} supported only at $x$.
  \end{enumerate}
\end{corollary}
\begin{proof}
  By Corollary \ref{C:vanishing-locus}, $F_{X-\{x\}} \simeq 0$, which
  gives \itemref{CI:is-sheaf:support}. For \itemref{CI:is-sheaf:bounded-above},
  Choose $e\geq 0$ such that $\trunc{>e}F \simeq 0$;
  then
  \begin{align*}
    \Hom_{\Orb_{X}}(\COHO{e}(F),\kappa(x,\zeta)) &= \Hom_{\Orb_{X}}(\COHO{e}(F)[-e],\kappa(x,\zeta)[-e]) \\
                                                 &\subseteq \Hom_{\Orb_X}(F,\kappa(x,\zeta)[-e]). 
  \end{align*}
  In particular, the right hand is zero if $e>0$. Since $F$ is
  pseudocoherent, $\COHO{e}(F)$ is a finite type $\Orb_X$-module. By
  Nakayama's Lemma, it suffices to prove that $i_x^*F \simeq 0$.  By
  descending induction, it follows that $\trunc{>0}F \simeq 0$.  For
  \itemref{CI:is-sheaf:bounded-below}, choose $g\geq 0$ such that
  $\trunc{<-g}F \simeq 0$; then
  \[
    \Hom_{\Orb_X}(F,\kappa(x,\zeta)[g]) \twoheadrightarrow
    \Hom_{\Orb_X}(\COHO{-g}(F)[g],\kappa(x,\zeta)[g]) =
    \Hom_{\Orb_X}(\COHO{-g}(F), \kappa(x,\zeta)).
  \]
  If $g>d$, then $ \Hom_{\Orb_X}(\COHO{-g}(F), \kappa(x,\zeta)) = 0$ for
  all $\zeta$. Since $\Orb_X$ is coherent and $F$ is pseudocoherent,
  $\COHO{-g}(F)$ is a coherent $\Orb_X$-module. Now apply Corollary
  \ref{C:nakayama} to conclude that if $g>d$, then $\COHO{-g}(F) = 0$.

  For \itemref{CI:is-sheaf:regular}: with $g$ chosen as in
  \itemref{CI:is-sheaf:bounded-below}, we note that
  \begin{align*}
    \Hom_{\Orb_X}(\trunc{>-g}F,\hat{x}_{*}j_{*}\Orb_{\spec k}[d+g+1]) &\simeq  \Hom_{\Orb_{\spec A}}(\trunc{>-g}(\hat{x}^*F),j_{*}\Orb_{\spec k}[d+g+1]),
  \end{align*}
  which is zero because $A$ has global dimension $d$ and $\hat{x}$ is flat and affine. Hence, there is a surjection:
  \[
    \Hom_{\Orb_X}(F,\hat{x}_{*}j_{*}\Orb_{\spec k}[d+g])
    \twoheadrightarrow
    \Hom_{\Orb_X}(\COHO{-g}(F)[g],\hat{x}_{*}j_{*}\Orb_{\spec k}[d+g]).
  \]
  But 
  \begin{align*}
    \Hom_{\Orb_X}(\COHO{-g}(F)[g],\hat{x}_{*}j_{*}\Orb_{\spec k}[d+g])&\simeq  \Hom_{\Orb_X}(\COHO{-g}(F),\hat{x}_{*}j_{*}\Orb_{\spec k}[d])\\
                                                               &\simeq \Hom_{\Orb_{\spec A}}(\hat{x}^*\COHO{-g}(F),j_{*}\Orb_{\spec k}[d])\\
                                                               &\simeq \Hom_{\Orb_{\spec k}}(\LDERF j^*\hat{x}^*\COHO{-g}(F),\Orb_{\spec k}[d])\\
                                                               &\simeq \Tor_d^A(\hat{x}^*\COHO{-g}(F),k).
  \end{align*}
  Since $A$ is a regular local ring of dimension $d$,
  $\Tor_d^A(\hat{x}^*\COHO{-g}(F),k) \neq 0$ as $\COHO{-g}(F)$ is supported
  only at $x$. Putting this all together, we must have $g=0$.
\end{proof}
\section{Proof of Theorem \ref{T:bo-field}}
We generally follow the reductions of \cite[\S4]{MR4280492} and
\cite[\S5]{MR1651025}, but must be more careful due to not necessarily
assuming that $k$ is algebraically closed. Our first step is to reduce
to the situation of $k=\bar{k}$. 
\begin{lemma}\label{L:bo-field-bc-alg-clo}
  Fix an algebraic closure $k \subseteq \bar{k}$. Assume that $\Phi_K$
  satisfies the conditions of Theorem \ref{T:bo-field}. Then
  $\Phi_{K_{\bar{k}}}$ satisfies the conditions of Theorem
  \ref{T:bo-field}.
\end{lemma}
\begin{proof}
  Let $x\colon \spec \ell_x \to X$ be a closed point. Fix a
  $k$-embedding $\ell_x \subseteq \bar{k}$. If
  $\bar{x} \colon \spec \bar{k} \to X_{\bar{k}}$ denotes the induced
  $\bar{k}$-point of $X_{\bar{k}}$, then as $k$ has characteristic
  $0$, Example \ref{E:reg-rep-bc} implies that
  $\Xi(x)_{\bar{k}} \simeq \Xi(\bar{x})^{[l_x:k]}$. Note that any
  closed point of $X_{\bar{k}}$ is of the form $\bar{x}$ for some $x$. 

  Now if $\bar{x} \nsim \bar{y}$ or $i \notin [0, \dim X]$, then for
  generalized closed points $(\bar{x},\bar{\xi})$,
  $(\bar{y},\bar{\zeta})$ of $X_{\bar{k}}$ it follows from Example
  \ref{E:regular-rep} that
  \[
    \Hom_{\Orb_{{Y}_{\bar{k}}}}
    (\Phi_{K_{\bar{k}}}(\kappa(\bar{x},\bar{\xi})),\Phi_{K_{\bar{k}}}(\kappa(\bar{y},\bar{\zeta}))[i])
  \]
  is a direct summand of
  \begin{align*}
    \Hom_{\Orb_{{Y}_{\bar{k}}}}&
                                 (\Phi_{K_{\bar{k}}}(\Xi(x)_{\bar{k}}),\Phi_{K_{\bar{k}}}(\Xi(y)_{\bar{k}})[i]) \\
                               &\simeq \Hom_{\Orb_Y}(\Phi_K(\Xi(x)),\Phi_K(\Xi(y))[i])_{\bar{k}}\\
                               &\simeq \bigoplus_{\sigma,\tau} \Hom_{\Orb_Y}(\Phi_K(\kappa(x,\sigma)),\Phi_K(\kappa(y,\tau))[i])^{m(x,\sigma)m(y,\tau)}\\
    &= 0,
  \end{align*}
  which gives the desired vanishing. Similarly, there is an induced isomorphism  
  \begin{align*}
    \Hom_{\Orb_X}(\Xi(x),\Xi(x)) &\simeq \Hom_{\Orb_X}(\oplus_\xi \kappa(x,\xi)^{m(x,\xi)},\oplus_\zeta \kappa(x,\zeta)^{m(x,\zeta)}) \\
                                 &\simeq \bigoplus_{\xi,\zeta} \Hom_{\Orb_X}(\kappa(x,\xi),\kappa(x,\zeta))^{m(x,\xi)m(x,\zeta)}\\
                                 &\simeq \bigoplus_{\xi,\zeta} \Hom_{\Orb_Y}(\Phi_K(\kappa(x,\xi)),\Phi_K(\kappa(x,\zeta)))^{m(x,\xi)m(x,\zeta)}\\
    &\simeq \Hom_{\Orb_Y}(\Phi_K(\Xi(x)),\Phi_K(\Xi(x))).
  \end{align*}
  It follows after extending scalars to $\bar{k}$ and applying Example
  \ref{E:reg-rep-bc} that
  \[
    \Hom_{\Orb_{X_{\bar{k}}}}(\Xi(\bar{x}),\Xi(\bar{x})) \simeq \Hom_{\Orb_{Y_{\bar{k}}}}(\Phi_{K_{\bar{k}}}(\Xi(\bar{x})),\Phi_{K_{\bar{k}}}(\Xi(\bar{x}))).
  \]
  By Example \ref{E:regular-rep} and functoriality, the result follows.
\end{proof}
By Proposition \ref{P:ff-everywhere}, we may thus assume henceforth that
$k=\bar{k}$.

We take $\Omega$ as in
Corollary \ref{C:spanning}; then $\Omega$ is a spanning class for
$\DCAT^b_{\COH}(X)$. Now \cite[Thm.~2.3]{MR1651025} implies that
$F \colon \DCAT^b_{\COH}(X) \to \DCAT^b_{\COH}(Y)$ is fully faithful
if and only if the induced map:
\[
\Hom_{\Orb_X}(\kappa(x,\xi),\kappa(y,\zeta)[i]) \to \Hom_{\Orb_Y}(  F(\kappa(x,\xi)),F(\kappa(y,\zeta))[i])
\]
is bijective for all $i\in \Z$ and generalized closed points $(x,\xi)$
and $(y,\zeta)$ of $X$. By adjunction, the bijectivity
above is equivalent to the bijectivity of:
\[
  \Hom_{\Orb_X}(\kappa(x,\xi),\kappa(y,\zeta)[i]) \to \Hom_{\Orb_Y}( G
  F(\kappa(x,\xi)),\kappa(y,\zeta)[i]),
\]
where $G$ is the left adjoint to $F$, which is the map induced by
$\varepsilon(x,\xi)\colon GF(\kappa(x,\xi)) \to \kappa(x,\xi)$---and it suffices
to show that this map is an isomorphism. By hypothesis,
\begin{align*}
  \Hom_{\Orb_X}(\kappa(x,\xi),\kappa(x,\xi)) &\simeq \Hom_{\Orb_Y}(F(\kappa(x,\xi)),F(\kappa(x,\xi)))\\
                                             &\simeq \Hom_{\Orb_X}(GF(\kappa(x,\xi)),\kappa(x,\xi)).
\end{align*}
Hence, the adjunction map is non-zero. By Corollary \ref{C:is-sheaf},
we know that $Q(x,\xi)=GF(\kappa(x,\xi))$ is a coherent sheaf
supported only at $x$. Since $\kappa(x,\xi)$ is simple (Remark
\ref{R:simple-kappa}), it follows that $\varepsilon(x,\xi)$ is
surjective. Now form the short exact sequence:
\[
  0 \to C(x,\xi) \to Q(x,\xi) \xrightarrow{\varepsilon(x,\xi)} \kappa(x,\xi) \to 0.
\]
Then we must show that $C(x,\xi) = 0$. By Corollary \ref{C:nakayama},
it suffices to show that the group $\Hom_{\Orb_X}(C(x,\xi),\kappa(x,\zeta))$ is $0$
for all $\zeta\in \widehat{\mathcal{G}}_x$. Applying $\Hom_{\Orb_X}(-,\kappa(x,\zeta))$ to the exact sequence above, we see that
\[
  \Hom_{\Orb_X}(C(x,\xi),\kappa(x,\zeta)) \simeq
  \ker(\Ext^1_{\Orb_X}(\kappa(x,\xi),\kappa(x,\zeta)) \to
  \Ext^1_{\Orb_X}(Q(x,\xi),\kappa(x,\zeta))).
\]
Hence, it remains to prove that
\[
  \Ext^1_{\Orb_X}(\kappa(x,\xi),\kappa(x,\zeta)) \to \Ext^1_{\Orb_X}(Q(x,\xi),\kappa(x,\zeta)) \simeq \Ext^1_{\Orb_Y}(F(\kappa(x,\xi)),F(\kappa(x,\zeta)))
\]
is injective for all closed points $x$ and $\xi$,
$\zeta\in \widehat{\mathcal{G}}_x$. Let $x\colon \spec k \to X$ be a
closed point (since $k$ is algebraically closed, every closed point is
of this form). Let $\Xi(x) = x_{*}\Orb_{\spec k}$; then Example
\ref{E:regular-rep} implies that
\[
  \Xi(x) = \oplus_{\xi\in \widehat{\mathcal{G}}_x} \kappa(x,\xi)^{\oplus m(x,\xi)},
\]
for some $m(x,\xi)> 0$. Then the commutative diagram:
\[
  \xymatrix@C-1pc{\Ext^1_{\Orb_X}(\Xi(x),\Xi(x)) \ar[r] \ar@{=}[d] & \Ext^1_{\Orb_Y}(F(\Xi(x)),F(\Xi(x))) \ar@{=}[d]\\
    \displaystyle{\bigoplus_{\xi,\zeta\in \hat{\mathcal{G}}_x}} \Ext^1_{\Orb_X}(\kappa(x,\xi),\kappa(x,\zeta))^{m(x,\xi)m(x,\zeta)}
    \ar[r] & \displaystyle{\bigoplus_{\xi,\zeta\in \hat{\mathcal{G}}_x}}
    \Ext^1_{\Orb_Y}(F(\kappa(x,\xi)),F(\kappa(x,\zeta)))^{m(x,\xi)m(x,\zeta)}}
\]
shows that it is sufficient to prove the top row is injective for all
closed points $x$ of $X$. We have the following useful Lemma.
\begin{lemma}\label{L:adjoint-split}
  Let $G \colon \mathcal{A} \leftrightarrows \mathcal{B} \colon F$ be
  an adjoint pair of functors. If $a$, $a' \in \mathcal{A}$, then
  \[
    \Hom_{\mathcal{B}}(F(a),F(a')) \to \Hom_{\mathcal{A}}(GF(a),GF(a'))
  \]
  is a split injection. In particular, the map
  \[
    \Hom_{\mathcal{A}}(a,a') \to \Hom_{\mathcal{B}}(F(a),F(a')) 
  \]
  is injective if and only if the map
  \[
    \Hom_{\mathcal{A}}(a,a') \to \Hom_{\mathcal{A}}(GF(a),GF(a'))
  \]
  is injective.
\end{lemma}
\begin{proof}
  By adjunction we have the composition
  \[
    \Hom_{\mathcal{B}}(F(a),F(a')) \to \Hom_{\mathcal{A}}(GF(a),GF(a')) \simeq \Hom_{\mathcal{B}}(F(a),FGF(a'))
  \]
  and the composition is a split injection induced by
  $F(a') \to FGF(a')$.
\end{proof}

By Lemma \ref{L:adjoint-split}, it suffices to prove that the induced map
\[
  T(x) \colon \Ext^1_{\Orb_X}(\Xi(x),\Xi(x)) \to \Ext^1_{\Orb_X}(GF(\Xi(x)), GF(\Xi(x)))
\]
is injective for all closed points $x$ of $X$. Since $F$ and $G$ are
Fourier--Mukai (Proposition \ref{P:adjoints-FM}), the composition $GF$
is isomorphic to $\Phi_Q$ for some kernel
$Q \in \DCAT^b_{\COH}(X\times_k X)$. Let $a$,
$b \colon X\times_k X \to X$ be the projections onto the first and
second factors, respectively. Let $y\colon \spec k \to X$ be a
closed point and form the following diagram:
\[
  \xymatrix@+1pc{\spec k \ar[d]_{{y}} & \ar[l] X \ar[d]^{(y,\mathrm{id})} &\\
  X & \ar[l]^-{a}  X \times_{k} X \ar[r]_-{b} & X. }
\]
Then
\begin{align*}
  \Phi_{Q}(\Xi({y})) &= \RDERF b_{\qcsubscript,*}(Q \tensor^{\LDERF}_{\Orb_{X \times_{k} X}} \LDERF a^*{y}_*\Orb_{\spec {k}})\\
                     &\simeq \RDERF b_{\qcsubscript,*}(Q \tensor^{\LDERF}_{\Orb_{X \times_{k} X}} \RDERF ({y},\mathrm{id})_{\qcsubscript,*}\Orb_{X})\\
                     &\simeq \RDERF b_{\qcsubscript,*}\RDERF ({y},\mathrm{id})_{\qcsubscript,*}\LDERF ({y},\mathrm{id})^*Q\\
                     &\simeq \LDERF (y,\mathrm{id})^*Q
\end{align*}
But
$\Phi_Q(\Xi(y))\simeq \oplus_{\zeta\in\hat{\mathcal{G}}_y}
Q(y,\zeta)^{\oplus m(y,\zeta)}$, which we have already seen is a
coherent sheaf supported only at $y$, so $\LDERF (y,\mathrm{id})^*Q$ is a coherent
sheaf. Now the following square is cartesian:
\[
  \xymatrix{X \ar[r]^-{(y,\mathrm{id})} \ar[d] & X \times_k X \ar[d]^{a} \\ \spec k \ar[r]_{y} & X}
\]
and $a$ is flat, so Lemma \ref{L:bridgeland-nakayama} implies that $Q$
is coherent sheaf on $X\times_k X$ that is flat over $X$ via $a$.

By Theorem \ref{T:anno-log}, the adjunction
$\Phi_Q \simeq GF \Rightarrow \mathrm{id}$ is induced by a morphism of
coherent $\Orb_{X\times_k X}$-modules
$\varepsilon \colon Q \to (\Delta_X)_*\Orb_X$. Also, $\LDERF (y,\mathrm{id})^*(\Delta_X)_*\Orb_X \simeq \Xi({y})$ and $(y,\mathrm{id})^*\varepsilon$ is equivalent to the adjunction
\[
\varepsilon({y}) \colon \Phi_{Q}(\Xi({y})) \to \Xi({y}).
\]
which we have already seen to be surjective (it is the direct sum of
the surjective maps $\varepsilon(y,\zeta)$). 
It follows immediately that $\varepsilon$ is surjective
and $C=\ker(Q \to (\Delta_X)_*\Orb_X)$ is a coherent
$\Orb_{X\times_k X}$-module, which is flat over $X$ via $a$. If
$\varepsilon$ is an isomorphism, then $T(x)$ is bijective for all $x$
and we're done. In fact, since $C$ is flat over $X$ and $X$ is
integral, it suffices to prove that $\LDERF (y,\mathrm{id})^*C = 0$
for a dense set of $y$ in $X$. In other words, we must prove that the adjunction  $\varepsilon(y)$ 
is an isomorphism for a dense set of $y$; equivalently, that $T(x)$ is
injective for a dense set of
$y$. We now proceed to arrange ourselves to be in the situation of Remark
\ref{R:KS}. 

By Proposition \ref{P:key}, the morphism
$X \xrightarrow{\underline{\Delta_*\Orb_X}} \underline{\COH}_{X/k}$ is
smooth. Let $W \subseteq \underline{\COH}_{X/k}$ be its image, which
is an open substack. Let $T \to W$ be a morphism, where $T$ is an
affine scheme; then there is a corresponding finitely presented
$\Orb_{T\times_k X}$-module $F$ that is flat over $T$. We claim that
$\Phi_{Q_T}(F)$ is a finitely presented $\Orb_{T\times_k X}$-module
that is flat over $T$. By Lemma \ref{L:fm-bc}, this is smooth-local on
$T$. In particular, since $X\times_{W} T \to T$ is a smooth
surjection, we may assume that $T \to W$ factors through $X \to
W$. Again by Lemma \ref{L:fm-bc}, we may thus assume that $T=X$ and
$F=\Delta_*\Orb_X$. In this case, however, we have
$Q_X \simeq \pi_{13}^*Q$ on
$(X\times_k X) \times_X (X\times_k X) \simeq X \times_k X \times_k X$
and $\Phi_{\pi_{13}^*Q}(\Delta_*\Orb_X) \simeq Q$, which we have
already seen is a coherent $\Orb_{X\times_S X}$-module that is flat
over $X$ via $p$. Hence, there is an induced morphism
\[
  \chi \colon W \to \underline{\COH}_{X/k} \colon (T \xrightarrow{\underline{F}} W) \mapsto
  (T \to W \to S, \Phi_{Q_T}(F))
\]
Since $W$ is of finite type over $k$ with affine diagonal and
$\underline{\COH}_{X/k}$ is locally of finite type with affine
diagonal, $\chi$ is of finite type with affine diagonal. It follows
immediately that the relative inertia of $\chi$ is affine and of
finite type.

We claim that $\chi$ is representable. To see this, we note that it
suffices to check this on finite type points over $k$. Since $k$ is an
algebraically closed field, finite type points are all closed
$k$-points. Let $x \colon \spec k \to X$ be a closed point of $X$;
then the induced point of $W$ corresponds to $\Xi({x})$. It remains to
prove that
$\Aut_{\Orb_{X}}(\Xi({x})) \to \Aut_{\Orb_{X}}(\Phi_{Q}(\Xi({x})))$ is
a monomorphism of group schemes over $\spec k$. Since $k$ is
algebraically closed and of characteristic $0$, it suffices to check
this induces an injection on
$k$-points. Now
\[
  \Aut_{\Orb_{X}}(\Xi({x}))(\spec k) = \Aut_{\Orb_{X}}(\Xi({x}))
  \subseteq \End_{\Orb_{X}}(\Xi({x})),
\]
so it remains to show
$\End_{\Orb_{X}}(\Xi({x}))
\subseteq \End_{\Orb_{X}}(\Phi_{Q}(\Xi({x})))$. Now we have already
seen that the Bondal--Orlov conditions imply that
\begin{align*}
  \Hom_{\Orb_X}(\Xi(x),\Xi(y))
  &\simeq \Hom_{\Orb_X}(F(\Xi(x)),F(\Xi(y))).
\end{align*}
By Lemma \ref{L:adjoint-split}, we
obtain the desired injectivity result.
Hence, $\chi$ is representable.

We now claim that $\chi$ is radiciel. Since
$\chi$ is representable and of finite type, it suffices to
prove that if $x_1$, $x_2$ are two $k$-points of $W$ and
$\chi(x_1) \simeq \chi(x_2)$ in $\underline{\COH}_{X/k}$, then
$x_1 \simeq x_2$ in $W$. We may assume that the $x_i$ lift to $X$. We have also
seen that $\chi(x_i)$ is equivalent to $\Phi_{Q}(\Xi({x}_i))$ on
$X$. But $\Phi_Q(\Xi({x}_i))$ is supported only at ${x}_i$, so
${x}_1 \simeq {x}_2$ in $X$. Hence,
$x_1 \simeq x_2$ in $W$ and we have the claim.

It now follows from Remark \ref{R:KS} and Proposition \ref{P:ks-functorial} that there is a dense set of closed points $x$ of $X$ such that the Kodaira--Spencer map:
\[
T(x) \colon   \Ext^1_{\Orb_{X}}(\Xi({x}),\Xi({x})) \to \Ext^1_{\Orb_{X}}(\Phi_{Q}(\Xi({x})),\Phi_{Q}(\Xi({x})))
\]
is injective, which completes the proof.  
\appendix
\section{Retracted covers}\label{A:retracted-covers}
\begin{proposition}\label{P:retract-cover}
  Let $X$ be a quasicompact algebraic stack with affine diagonal and
  the resolution property. Then $X$ is cohomologically affine if and
  only if there is a smooth surjection $p \colon \spec A \to X$ such
  that $\Orb_X \to p_{*}\Orb_{\spec A}$ admits a retraction.  
\end{proposition}
\begin{proof}
  For the necessity: let $p$ be such a cover and let $M \in
  \QCOH(X)$. Then
  $M \to p_{\qcsubscript,*}p^*M \simeq p_{\qcsubscript,*}\Orb_{\spec A} \tensor_{\Orb_X} M$ admits a
  retraction. In particular, $H^i(X,M)$ is a direct summand of
  $H^i(X,p_{\qcsubscript,*}p^*M) \simeq H^i(\spec A,p^*M) = 0$, whenever
  $i>0$. Hence, $X$ is cohomologically affine. For the sufficiency:
  the Totaro--Gross Theorem \cite{MR2108211,2013arXiv1306.5418G} shows
  that $X\simeq [\spec A/\GL_{n,\Z}]$ for some $n$ and ring $A$. Let
  $p \colon \spec A \to X$ be the induced covering, which we claim has
  the desired property. Indeed, let
  $q \colon \spec \Z \to B\GL_{n,\Z}$ be the usual covering. Observe
  that $q_{\qcsubscript,*}\Orb_{\spec \Z} = \cup_{m\geq n} V_m$, where
  $V_m = \det^{-m}W_m \subseteq
  \Z[\GL_{n,\Z}]=\Z[\{x_{ij}\}_{i,j=1}^n]_{\det}$ and
  \[
    W_m = \{ f \in \Z[\{x_{ij}\}_{i,j=1}^n] \suchthat \deg f \leq m\}.
  \]
  Then the $V_m$ are finite rank vector bundles on $B\GL_{n,\Z}$, as
  are the quotients $V_{m+1}/V_m$. Now
  $\Orb_{B\GL_{n,\Z}} \subseteq V_m$ and
  $\Orb_{\spec \Z} \subseteq p^*V_m$ splits via picking off the
  constant term; in particular, $Q_m = V_m/\Orb_{B\GL_{n,\Z}}$ is a
  finite rank vector bundle for all $m$. Similarly, the quotients
  $Q_{m+1}/Q_m \simeq V_{m+1}/V_m$ are finite rank vector bundles on
  $B\GL_{n,\Z}$ for all $m$. Let $a\colon X \to B\GL_{n,\Z}$ be the
  induced morphism. Let $\tilde{Q}_m = a^*Q_m$; then there is an exact
  sequence
  \[
    0 \to \Orb_X \to p_{\qcsubscript,*}\Orb_{\spec A} \to \cup_m \tilde{Q}_m \to 0.
  \]
  It suffices to prove that $\Ext^1_{\Orb_X}(\tilde{Q},\Orb_X) = 0$,
  where $\tilde{Q} = \cup_m \tilde{Q}_m$, when $X$ is cohomologically
  affine. To this end, we note that $\tilde{Q} \simeq \mathrm{hocolim}_m \tilde{Q}_m$ and so
  \[
    \Ext^1_{\Orb_X}(\tilde{Q},\Orb_X) = \COHO{1}(\mathrm{holim}_m\RHom_{\Orb_X}(\tilde{Q}_m,\Orb_X)).
  \]
  We have the Milnor exact sequence, however:
  \[
    0 \to {\varprojlim_m}^1 \Hom_{\Orb_X}(\tilde{Q}_m,\Orb_X) \to
    \Ext^1_{\Orb_X}(\tilde{Q},\Orb_X) \to \varprojlim_m
    \Ext^1_{\Orb_X}(\tilde{Q}_m,\Orb_X) \to 0.
  \]
  But
  \[
    \varprojlim_m
    \Ext^1_{\Orb_X}(\tilde{Q}_m,\Orb_X) \simeq \varprojlim_m
    H^1(X,\tilde{Q}_m^\vee) = 0,
  \]
  and
  \[
    H^0(X,\tilde{Q}_{m+1}^\vee) =
    \Hom_{\Orb_X}(\tilde{Q}_{m+1},\Orb_X) \to
    \Hom_{\Orb_X}(\tilde{Q}_m,\Orb_X) = H^0(X,\tilde{Q}_m^\vee)
  \]
  is surjective because $H^0(X,-)$ is exact on quasicoherent sheaves
  and $\tilde{Q}_{m+1}^\vee \to \tilde{Q}_m^{\vee}$ is surjective
  because $\tilde{Q}_{m+1}/\tilde{Q}_m$ is a finite rank vector
  bundle. Hence, we have the desired vanishing.
\end{proof}
\begin{corollary}\label{C:loc-proj-inj-glob}
  Let $X$ be a cohomologically affine, quasicompact algebraic stack
  with affine diagonal and the resolution property. Then locally
  projective quasicoherent $\Orb_X$-modules are projective. If $X$ is
  noetherian, then locally injective quasicoherent $\Orb_X$-modules
  are injective.
\end{corollary}
\begin{proof}
  Let $M \in \QCOH(X)$ be locally projective. It suffices to prove
  that $\Ext^i_{\Orb_X}(M,N) = 0$ for all $N\in \QCOH(X)$ and
  $i>0$. Similarly, if $N\in \QCOH(X)$ is locally injective, then it
  suffices to prove that $\Ext^i_{\Orb_X}(M,N) = 0$ for all
  $M \in \QCOH(X)$ and $i>0$. Let $p \colon \spec A \to X$ be a smooth
  covering as in Proposition \ref{P:retract-cover}; then
  $N \to p_{\qcsubscript,*}p^*N \simeq p_{\qcsubscript,*}\Orb_{\spec A} \tensor_{\Orb_X} N$ admits a
  retract. In particular, $\Ext^i_{\Orb_X}(M,N)$ is a direct summand
  of
  $\Ext^i_{\Orb_X}(M,p_{\qcsubscript,*}p^*N) \simeq \Ext^i_{\Orb_{\spec
      A}}(p^*M,p^*N)$. Hence, we are reduced to the situation where
  $X=\spec A$ is an affine scheme. If $M$ is locally projective over
  $\spec A$, then it is projective \cite{MR0308104}. If $N$ is locally
  injective over $\spec A$, then there is a faithfully flat covering
  $q \colon \spec A' \to \spec A$ such that $q^*N$ is injective. By
  Baer's Criterion \cite[Tag \spref{05NU}]{stacks-project}, it
  suffices to prove that
  $\Hom_{\Orb_X}(\Orb_X,N) \to \Hom_{\Orb_X}(I,N)$ is surjective for
  all coherent ideals $I \subseteq \Orb_X$. Since $A \to A'$ is
  faithfully flat, it suffices to prove that
  $\Hom_{\Orb_X}(\Orb_X,N)\tensor_A A' \to \Hom_{\Orb_X}(I,N)\tensor_A
  A'$ is surjective.  But $\Orb_X$ and $I$ are finitely presented
  $\Orb_X$-modules (the latter because $X$ is noetherian), so the
  formation of $\Hom$ commutes with base change. Hence, it remains to
  prove that
  $\Hom_{\Orb_{\spec A'}}(\Orb_{\spec A'},q^*N) \to \Hom_{\Orb_{\spec
      A'}}(q^*I,q^*N)$ is surjective, which follows from the
  injectivity of $q^*N$.
\end{proof}
\section{Strong generators for tame stacks}\label{A:strong-gens}
For an efficient introduction to strong generators and descendable
morphisms, we refer the reader to \cite{aoki2020quasiexcellence}. Also
see \cite{MR2434186}. Our main result in this appendix is the
following theorem, which generalizes the main result of
\cite{aoki2020quasiexcellence} from schemes to tame algebraic stacks.
\begin{theorem}\label{T:strong-generator-tame}
  If $X$ is a noetherian, separated, quasiexcellent tame
  algebraic stack of Krull finite dimension, then
  $\DCAT^b_{\COH}(X)$ has a strong generator.
\end{theorem}
\begin{proof}
  By \cite[Thm.~B]{rydh-2009}, there exists a finite and surjective
  morphism $f \colon W \to X$, where $W$ is a scheme. It follows that
  $W$ is a noetherian, separated, quasiexcellent scheme of finite
  Krull dimension. By \cite[Cor.~4.5]{aoki2020quasiexcellence}, it
  suffices to prove that the morphism $f$ is descendable, in the sense
  of \cite{MR3459022}. This is just Lemma \ref{L:descendable-finite}
  and the result follows.
\end{proof}
\begin{lemma}\label{L:descendable-finite}
  Let $X$ be a quasicompact and quasiseparated and tame algebraic stack. If
  $f\colon W \to X$ is a finite, finitely presented, and surjective morphism, then $f$ is
  descendable.
\end{lemma}
\begin{proof}
  By absolute noetherian approximation \cite[Thm.~D]{rydh-2009}, we
  may assume that $X$ is noetherian.  We now prove the result by
  noetherian induction on the closed substacks of $X$; that is, we
  will assume that for every closed immersion $Z \subsetneq X$ the
  restriction $f^{-1}(Z) \to Z$ is descendable. Let
  $I \subseteq \Orb_X$ be the nilradical; then
  $X_{\mathrm{red}} \to X$ is descendable as $I$ is nilpotent. As the
  composition of descendable morphisms is descendable, we may further
  assume that $X$ is reduced. By generic flatness, there is a dense
  open $U \subseteq X$ such that $f^{-1}(U) \to U$ is flat. Let
  $U \to U_{\mathrm{tame}}$ be the tame moduli space of $U$; then
  passing to a dense affine open of $U_{\mathrm{tame}}$, we may
  further shrink $U$ so that $U$ and $f^{-1}(U)$ are cohomologically
  affine. Then the morphism $\Orb_U \to (f_U)_{\qcsubscript,*}\Orb_{f^{-1}(U)}$ is
  split: $f^{-1}(U) \to U$ is finite and faithfully flat, so
  $Q_U = \coker(\Orb_U \to (f_U)_{\qcsubscript,*}(\Orb_{f^{-1}(U)})$ is a vector
  bundle of finite rank on $U$ and
  $\Ext^1_{\Orb_U}(Q_U,\Orb_U) = H^1(U,Q_U^\vee) = 0$. In particular,
  in the distinguished triangle
  \[
    Q[-1] \xrightarrow{\delta} \Orb_X \to f_{\qcsubscript,*}\Orb_W \to Q,
  \]
  we see that $\delta_U = 0$. That is,
  $Q[-1] \xrightarrow{\delta} \Orb_X \to j_{\qcsubscript,*}\Orb_U$ is $0$. Note that
  $j\colon U \to X$ is affine because $X$ has affine diagonal and $U$
  is cohomologically affine. By \cite[Thm.~1.3.8]{MR3014449},
  $j_{\qcsubscript,*}\Orb_U \simeq \RDERF j_{\qcsubscript,*}\Orb_U \simeq \hocolim{n}
  \SRHom_{\Orb_X}(J^n,\Orb_X)$, where $J \subseteq \Orb_X$ is a
  coherent ideal such that $X-V(J) = U$. Since $Q \in \COH(X)$ and
  $\SRHom_{\Orb_X}(J^n,\Orb_X) \in \DQCOH^{\geq 0}(X)$, it follows
  from \cite[Lem.~1.2(3)]{perfect_complexes_stacks} that we can
  arrange that $Q[-1] \to \Orb_X \to \SRHom_{\Orb_X}(J^n,\Orb_X)$ is
  $0$. We may now replace $J$ by $J^n$ and by adjunction, the
  composition $Q[-1] \otimes^{\LDERF} J \to Q[-1] \to \Orb_X$ is
  $0$. In particular, $\delta$ factors through
  $\bar{\delta} \colon Q[-1] \otimes^{\LDERF}_{\Orb_X} \Orb_X/J \to
  \Orb_X$. Now form the short exact sequence:
  \[
    0 \to \Orb_X/J' \to f_{\qcsubscript,*}(\Orb_W/J\Orb_W) \to Q' \to 0.
  \]
  There is consequently a morphism of triangles:
  \[
    \xymatrix{Q[-1]\otimes^{\LDERF}_{\Orb_X} \Orb_X/J\ar[d] \ar[r]^-{\bar{\delta}} & O_X/J \ar[r] \ar[d] & (f_{\qcsubscript,*}\Orb_W)\otimes^{\LDERF}_{\Orb_X} \Orb_X/J \ar[r] \ar[d] & Q\otimes^{\LDERF}_{\Orb_X} \Orb_X/J \ar[d]\\
      Q'[-1] \ar[r]^{\delta'} & \Orb_X/J' \ar[r] & f_{\qcsubscript,*}(\Orb_W/J\Orb_W)
      \ar[r] & Q',}
  \]
  By noetherian induction, $\delta'^{\otimes r} = 0$ for some
  $r \gg 0$. It follows that there is a factorization of
  $\bar{\delta}^{\otimes r} \colon (Q[-1]\otimes^{\LDERF}_{\Orb_X}
  \Orb_X/J)^{\otimes r} \to \Orb_X/J$ through
  $J'/J \subseteq \Orb_X/J$. Since $f$ is finite and surjective,
  $J'/J$ is a nilpotent ideal in $\Orb_X/J$. Hence, there is
  an $s$ such that $(J'/J)^{\otimes s} \to \Orb_X/J$ is the $0$
  map. Putting this together we see that $\bar{\delta}^{\otimes rs}$
  is $0$ and so $\delta^{\otimes rs}$ is too, which completes the proof.
\end{proof}
\begin{corollary}\label{C:adjoints-tame}
  Let $R$ be a quasiexcellent ring of finite Krull dimension. Let
  $p \colon X \to \spec R$ be a proper morphism of tame algebraic
  stacks. Let $F \colon \DCAT^b_{\COH}(X) \to \mathsf{T}$ be an
  $R$-linear functor of triangulated categories such that $\mathsf{T}$
  is proper over $R$ (i.e.,
  $\oplus_{n\in \Z} \Hom_{\mathsf{T}}(t,t'[n])$ is a finitely
  generated $R$-module for $t$, $t'\in \mathsf{T}$). Then $F$ admits a
  right adjoint $F_\rho$. If $X$ is regular, then $F$ also admits a
  left adjoint $F_\lambda$.
\end{corollary}

\begin{proof}[Proof of Corollary \ref{C:adjoints-tame}]
  Fix $t\in \mathsf{T}$. Then Theorem \ref{T:strong-generator-tame}
  and \cite[Cor.~4.18]{MR2434186} implies that the functor
  \[
    \Hom_{\mathsf{T}}(F(-),t) \colon \DCAT^b_{\COH}(X) \to \MOD(R)
  \]
  is representable by some $F_\rho(t) \in \DCAT^b_{\COH}(X)$. The
  assignment $t \mapsto F_\rho(t)$ defines a right adjoint to $F$. For
  the left adjoint: $X$ is regular so $\DCAT^b_{\COH}(X) = \PERF(X)$
  and every object is dualizable. In particular, we see that the functor 
  \[
    \Hom_{\mathsf{T}}(t,F((-)^\vee)) \colon
    \DCAT^b_{\COH}(X)^{\opp} \to \MOD(R)
  \]
  is representable by some $\tilde{t} \in \DCAT^b_{\COH}(X)$. Hence, if $x\in \DCAT^b_{\COH}(X)$, then
  \[
    \Hom_{\mathsf{T}}(t,F(x)) \simeq
    \Hom_{\mathsf{T}}(t,F(x^{\vee\vee})) \simeq \Hom_{\Orb_X}(x^\vee,
    \tilde{t}) \simeq \Hom_{\Orb_X}(\tilde{t}^\vee, x).
  \]
  The assignment $t\mapsto \tilde{t}^\vee$ gives the left adjoint.
\end{proof}

\bibliography{bibtex_db/references}

\providecommand{\MR}{\relax\ifhmode\unskip\space\fi MR }
\providecommand{\MRhref}[2]{%
  \href{http://www.ams.org/mathscinet-getitem?mr=#1}{#2}
}
\providecommand{\href}[2]{#2}
\begin{thebibliography}{HRLMS09}

\bibitem[AHR19]{etale_local_stacks}
J.~Alper, J.~Hall, and D.~Rydh, \emph{The \'etale local structure of algebraic
  stacks}, 2019,
  \href{http://arXiv.org/abs/1912.06162}{\mbox{arXiv:1912.06162}}.

\bibitem[AL12]{MR2964634}
R.~Anno and T.~Logvinenko, \emph{On adjunctions for {F}ourier-{M}ukai
  transforms}, Adv. Math. \textbf{231} (2012), no.~3-4, 2069--2115.

\bibitem[Alp13]{2008arXiv0804.2242A}
J.~Alper, \emph{{Good moduli spaces for Artin stacks}}, Ann. Inst. Fourier
  (Grenoble) \textbf{63} (2013), no.~6, 2349--2402.

\bibitem[Aok21]{aoki2020quasiexcellence}
K.~Aoki, \emph{Quasiexcellence implies strong generation}, J. Reine Angew.
  Math. \textbf{780} (2021), 133--138.

\bibitem[BB03]{MR1996800}
A.~Bondal and M.~Van~den Bergh, \emph{Generators and representability of
  functors in commutative and noncommutative geometry}, Mosc. Math. J.
  \textbf{3} (2003), no.~1, 1--36, 258.

\bibitem[BL24]{bragg2024murphys}
D.~Bragg and M.~Lieblich, \emph{Murphy's law for algebraic stacks}, 2024,
  \href{http://arXiv.org/abs/2402.00862}{\mbox{arXiv:2402.00862}}.

\bibitem[BO95]{bo-semiorthogonal}
A.~Bondal and D.~Orlov, \emph{Semiorthogonal decomposition for algebraic
  varieties}, 1995.

\bibitem[Bri99]{MR1651025}
T.~Bridgeland, \emph{Equivalences of triangulated categories and
  {F}ourier-{M}ukai transforms}, Bull. London Math. Soc. \textbf{31} (1999),
  no.~1, 25--34.

\bibitem[BS13]{MR3014449}
M.~P. Brodmann and R.~Y. Sharp, \emph{Local cohomology}, second ed., Cambridge
  Studies in Advanced Mathematics, vol. 136, Cambridge University Press,
  Cambridge, 2013, An algebraic introduction with geometric applications.

\bibitem[BS17]{MR3674218}
B.~Bhatt and P.~Scholze, \emph{Projectivity of the {W}itt vector affine
  {G}rassmannian}, Invent. Math. \textbf{209} (2017), no.~2, 329--423.

\bibitem[BZFN10]{MR2669705}
D.~Ben-Zvi, J.~Francis, and D.~Nadler, \emph{Integral transforms and {D}rinfeld
  centers in derived algebraic geometry}, J. Amer. Math. Soc. \textbf{23}
  (2010), no.~4, 909--966.

\bibitem[CKN01]{MR1867248}
J.~D. Christensen, B.~Keller, and A.~Neeman, \emph{Failure of {B}rown
  representability in derived categories}, Topology \textbf{40} (2001), no.~6,
  1339--1361.

\bibitem[Gro17]{2013arXiv1306.5418G}
P.~Gross, \emph{Tensor generators on schemes and stacks}, Algebr. Geom.
  \textbf{4} (2017), no.~4, 501--522.

\bibitem[Hal17]{MR3589351}
J.~Hall, \emph{Openness of versality via coherent functors}, J. Reine Angew.
  Math. \textbf{722} (2017), 137--182.

\bibitem[Hal22]{hall2022remarks}
J.~Hall, \emph{Further remarks on derived categories of algebraic stacks},
  2022, \href{http://arXiv.org/abs/2205.09312}{\mbox{arXiv:2205.09312}}.

\bibitem[Hal23]{GAGA_theorems}
J.~Hall, \emph{G{AGA} theorems}, J. Math. Pures Appl. (9) \textbf{175} (2023),
  109--142.

\bibitem[HR15]{hallj_dary_alg_groups_classifying}
J.~Hall and D.~Rydh, \emph{Algebraic groups and compact generation of their
  derived categories of representations}, Indiana Univ. Math. J. \textbf{64}
  (2015), no.~6, 1903--1923.

\bibitem[HR17]{perfect_complexes_stacks}
J.~Hall and D.~Rydh, \emph{Perfect complexes on algebraic stacks}, Compositio
  Math. \textbf{153} (2017), no.~11, 2318--2367.

\bibitem[HR18]{etale_dev_add}
J.~Hall and D.~Rydh, \emph{Addendum to ``\'{E}tale d\'evissage, descent and
  pushouts of stacks'' [{J}. {A}lgebra 331 (1) (2011) 194--223] [
  {MR}2774654]}, J. Algebra \textbf{498} (2018), 398--412.

\bibitem[HR23]{mayer-vietoris}
J.~Hall and D.~Rydh, \emph{Mayer-{V}ietoris squares in algebraic geometry}, J.
  Lond. Math. Soc. (2) \textbf{107} (2023), no.~5, 1583--1612.

\bibitem[HRLMS07]{MR2323539}
D.~Hern\'{a}ndez~Ruip\'{e}rez, A.~C. L\'{o}pez~Mart\'{\i}n, and F.~S. de~Salas,
  \emph{Fourier-{M}ukai transforms for {G}orenstein schemes}, Adv. Math.
  \textbf{211} (2007), no.~2, 594--620.

\bibitem[HRLMS09]{MR2505443}
D.~Hern\'{a}ndez~Ruip\'{e}rez, Ana~C. L\'{o}pez~M., and F.~Sancho~de Salas,
  \emph{Relative integral functors for singular fibrations and singular
  partners}, J. Eur. Math. Soc. (JEMS) \textbf{11} (2009), no.~3, 597--625.

\bibitem[Kie72]{MR0382280}
R.~Kiehl, \emph{Ein ``{D}escente''-{L}emma und {G}rothendiecks
  {P}rojektionssatz f\"ur nichtnoethersche {S}chemata}, Math. Ann. \textbf{198}
  (1972), 287--316.

\bibitem[KM97]{MR1432041}
S.~Keel and S.~Mori, \emph{Quotients by groupoids}, Ann. of Math. (2)
  \textbf{145} (1997), no.~1, 193--213.

\bibitem[Lie06]{MR2177199}
M.~Lieblich, \emph{Moduli of complexes on a proper morphism}, J. Algebraic
  Geom. \textbf{15} (2006), no.~1, 175--206.

\bibitem[LMB]{MR1771927}
G.~Laumon and L.~Moret-Bailly, \emph{Champs alg\'ebriques}, Ergebnisse der
  Mathematik und ihrer Grenzgebiete. 3. Folge., vol.~39, Springer-Verlag,
  Berlin, 2000.

\bibitem[LP21]{MR4280492}
B.~Lim and A.~Polishchuk, \emph{Bondal-{O}rlov fully faithfulness criterion for
  {D}eligne-{M}umford stacks}, Manuscripta Math. \textbf{165} (2021), no.~3-4,
  469--481.

\bibitem[Mat16]{MR3459022}
A.~Mathew, \emph{The {G}alois group of a stable homotopy theory}, Adv. Math.
  \textbf{291} (2016), 403--541.

\bibitem[Nee96]{MR1308405}
A.~Neeman, \emph{The {G}rothendieck duality theorem via {B}ousfield's
  techniques and {B}rown representability}, J. Amer. Math. Soc. \textbf{9}
  (1996), no.~1, 205--236.

\bibitem[{Nee}18]{Neeman_Approx}
A.~{Neeman}, \emph{{Triangulated categories with a single compact generator and
  a Brown representability theorem}}, April 2018,
  \href{http://arXiv.org/abs/1804.02240}{\mbox{arXiv:1804.02240}}.

\bibitem[Nee21]{MR4239176}
A.~Neeman, \emph{New progress on {G}rothendieck duality, explained to those
  familiar with category theory and with algebraic geometry}, Bull. Lond. Math.
  Soc. \textbf{53} (2021), no.~2, 315--335.

\bibitem[Nee23]{MR4575464}
A.~Neeman, \emph{An improvement on the base-change theorem and the functor
  {$f^!$}}, Bull. Iranian Math. Soc. \textbf{49} (2023), no.~3, Paper No. 25,
  163.

\bibitem[Ols05]{MR2183251}
M.~Olsson, \emph{On proper coverings of {A}rtin stacks}, Adv. Math.
  \textbf{198} (2005), no.~1, 93--106.

\bibitem[RG71]{MR0308104}
M.~Raynaud and L.~Gruson, \emph{Crit\`eres de platitude et de projectivit\'e.
  {T}echniques de ``platification'' d'un module}, Invent. Math. \textbf{13}
  (1971), 1--89.

\bibitem[Rou08]{MR2434186}
R.~Rouquier, \emph{Dimensions of triangulated categories}, J. K-Theory
  \textbf{1} (2008), no.~2, 193--256.

\bibitem[Ryd11]{MR2774654}
D.~Rydh, \emph{\'{E}tale d\'evissage, descent and pushouts of stacks}, J.
  Algebra \textbf{331} (2011), 194--223.

\bibitem[Ryd13]{MR3084720}
D.~Rydh, \emph{Existence and properties of geometric quotients}, J. Algebraic
  Geom. \textbf{22} (2013), no.~4, 629--669.

\bibitem[Ryd15]{rydh-2009}
D.~Rydh, \emph{Noetherian approximation of algebraic spaces and stacks}, J.
  Algebra \textbf{422} (2015), 105--147.

\bibitem[Ryd16]{rydh-2014}
D.~Rydh, \emph{Approximation of sheaves on algebraic stacks}, Int. Math. Res.
  Not. \textbf{2016} (2016), no.~3, 717--737.

\bibitem[Ryd23]{rydh_absolute_approximation}
D.~Rydh, \emph{Absolute noetherian approximation of algebraic stacks}, 2023,
  \href{http://arXiv.org/abs/2311.09208}{\mbox{arXiv:2311.09208}}.

\bibitem[SAG]{lurie_sag}
J.~Lurie, \emph{{S}pectral {A}lgebraic {G}eometry}, available on homepage, Oct
  2016.

\bibitem[SGA6]{MR0354655}
\emph{Th\'eorie des intersections et th\'eor\`eme de {R}iemann-{R}och}, Lecture
  Notes in Mathematics, Vol. 225, Springer-Verlag, Berlin, 1971, S{\'e}minaire
  de G{\'e}om{\'e}trie Alg{\'e}brique du Bois-Marie 1966--1967 (SGA 6),
  Dirig{\'e} par P. Berthelot, A. Grothendieck et L. Illusie. Avec la
  collaboration de D. Ferrand, J. P. Jouanolou, O. Jussila, S. Kleiman, M.
  Raynaud et J. P. Serre.

\bibitem[Stacks]{stacks-project}
The {Stacks Project Authors}, \emph{{S}tacks {P}roject},
  \url{http://stacks.math.columbia.edu}.

\bibitem[Tot04]{MR2108211}
B.~Totaro, \emph{The resolution property for schemes and stacks}, J. Reine
  Angew. Math. \textbf{577} (2004), 1--22.

\bibitem[Web22]{MR4479830}
R.~Webb, \emph{The moduli of sections has a canonical obstruction theory},
  Forum Math. Sigma \textbf{10} (2022), Paper No. e78, 47.

\end{thebibliography}
\bibliographystyle{bibtex_db/dary}
\end{document}